\documentclass[a4paper,12pt]{article}

\makeatletter
\@addtoreset{footnote}{page}
\makeatother

\usepackage{style-enumitem}

\usepackage{amsthm}
\usepackage{amsmath,amssymb,latexsym,amsfonts,mathrsfs}

\renewcommand{\hat}{\widehat}
\renewcommand{\tilde}{\widetilde}
\renewcommand{\bar}{\overline}
\newcommand{\un}[1]{\underline{#1}}

\newcommand{\absol}[1]{\left| #1 \right|} 
\newcommand{\rbra}[1]{\!\left( #1 \right)} 
\newcommand{\cbra}[1]{\!\left\{ #1 \right\}} 
\newcommand{\sbra}[1]{\!\left[ #1 \right]} 

\newcommand{\bN}{\ensuremath{\mathbb{N}}}

\newcommand{\bR}{\ensuremath{\mathbb{R}}}

\newcommand{\cB}{\ensuremath{\mathcal{B}}}

\newcommand{\cF}{\ensuremath{\mathcal{F}}}

\newcommand{\cK}{\ensuremath{\mathcal{K}}}

\newcommand{\cN}{\ensuremath{\mathcal{N}}}

\newcommand{\cT}{\ensuremath{\mathcal{T}}}
\newcommand{\cU}{\ensuremath{\mathcal{U}}}
\newcommand{\cV}{\ensuremath{\mathcal{V}}}
\newcommand{\cW}{\ensuremath{\mathcal{W}}}

\newcommand{\bfE}{\ensuremath{{\mathbf{E}}}}

\newcommand{\bfP}{\ensuremath{{\mathbf{P}}}}

\newcommand{\bfb}{\ensuremath{{\mathbf{b}}}}

\newcommand{\bfq}{\ensuremath{{\mathbf{q}}}}
\newcommand{\bfr}{\ensuremath{{\mathbf{r}}}}

\newcommand{\frq}{\ensuremath{{\mathfrak{q}}}}

\newcommand{\rmP}{\ensuremath{{\mathrm{P}}}}

\newcommand{\diff}{\ensuremath{{\mathrm{d}}}}

\theoremstyle{plain}
\newtheorem{Thm}{Theorem}[section]

\newtheorem{Lem}[Thm]{Lemma}
\newtheorem{Prop}[Thm]{Proposition}
\newtheorem{Cor}[Thm]{Corollary}

\theoremstyle{definition}

\newtheorem{Rem}[Thm]{Remark}

\numberwithin{equation}{section}

\makeatletter
\renewcommand\section{\@startsection {section}{1}{\z@}%
                                   {-3.5ex \@plus -1ex \@minus -.2ex}%
                                   {2.3ex \@plus.2ex}%
                                   {\normalfont\large\bf}}
\makeatother

\makeatletter
\renewcommand\subsection{\@startsection {subsection}{1}{\z@}%
                                   {-3.5ex \@plus -1ex \@minus -.2ex}%
                                   {2.3ex \@plus.2ex}%
                                   {\normalfont\normalsize\bf}}
\makeatother

\usepackage{color}
\allowdisplaybreaks[3]
\usepackage{mathtools}
\mathtoolsset{showonlyrefs=true}

\newcommand{\pcb}{{\rm pcb}}

\begin{document}

\begin{center}
{\Large \bf 
Stochastic control with dividend payments and capital injections for Markov additive processes
}
\end{center}
\begin{center}
Kei Noba
\end{center}

\begin{abstract}
Motivated by de Finetti's optimal dividend problem with capital injections, we study a stochastic control problem for the additive component of a Markov additive process (MAP). 
In contrast to previous studies, the modulating component is allowed to be a general right process on a Radon space, so the model is not restricted to finite-state regime switching and cannot in general be reduced to a finite collection of L\'evy process control problems. 
Capital injections are allowed at arbitrary times. 
We first consider the case in which dividend payments are allowed only at prescribed discrete times and establish necessary and sufficient conditions for the optimality of a strategy. 
These conditions then yield the optimality of a class of Markov-modulated periodic--classical barrier strategies. 
Combining this optimality result with an approximation argument, we obtain insight into the possible 
form of optimal strategies in the case where dividend payments, like capital injections, may be made at arbitrary times.
Because of the generality of the MAPs considered here, the proof techniques used in previous studies of similar problems are not directly applicable. 
We therefore develop an alternative argument based on the additive structure of MAPs and dynamic programming between dividend opportunities. 
The argument also suggests a possible approach to other stochastic control problems involving general MAPs.

\end{abstract}

\section{Introduction}
This paper studies a stochastic control problem for the additive component of a Markov additive process (MAP), motivated by de Finetti's optimal dividend problem with capital injections. 
In the financial interpretation, a firm pays dividends to its shareholders and may receive costly capital injections to avoid bankruptcy. The objective is to determine a dividend and capital injection strategy that maximizes the expected net present value (NPV) of dividends paid minus the cost of capital injections. 
We first consider a setting in which capital injections can be made at any time, while dividend payments are allowed only at renewal times with independent and identically distributed interarrival times. We refer to this case as the periodic--classical setting. We then use an approximation argument to examine the classical--classical setting, where dividend payments can be made at arbitrary times. For the periodic--classical setting, we establish necessary and sufficient conditions for optimality and prove the optimality of a class of Markov-modulated periodic--classical barrier (MMPCB) strategies. By the approximation, we obtain insight into 
the possible form 
of optimal strategies in the classical--classical setting.
\par
In this paper, we assume that the firm's uncontrolled capital process is given by the additive component of a MAP. 
A MAP is, roughly speaking, a right process consisting of two components, $X=\{X_t:t\geq0\}$, the additive component, and $Y=\{Y_t:t\geq0\}$, the modulating component. 
We assume that $X$ takes values in $\bR$ and $Y$ takes values in a Radon space $(E, \cB(E))$. 
The essential feature of a MAP is that it satisfies the following additive property:
for $(x, y)\in \bR \times E$, any non-negative measurable function $f$ and $t\geq 0$, 
\begin{align}
\bfE_{(x,y)} \sbra{f(X_{t}, Y_{t}) 
}
=
\bfE_{(0, y)} \sbra{f(X_{t}+x, Y_{t}) 
}.   \label{200_01}
\end{align}
Here, $\bfE_{(x,y)} $ denotes expectation under $\bfP_{(x, y)}$. 
In \cite{Cin1972a, Cin1972b}, 
MAPs were defined in a more general setting, and several fundamental properties were established.
When considering $Y$ alone, $Y$ itself is a right process, and the dynamics of $X$ depend on the behavior of $Y$.
MAPs can be regarded as a generalization of L\'evy processes. Indeed, if $E$ consists of a single point, then $X$ is a L\'evy process. 
Furthermore, when $E$ is a finite state space, $Y$ is a continuous-time Markov chain and we can construct $X$ by adding jumps and concatenating L\'evy processes whenever $Y$ changes state (for details, see, e.g. \cite[Chapter XI]{Asm2003} or \cite[Chapter 11]{KypPar2022}).
This viewpoint is made precise by Proposition 2.20 and Theorem 2.22 of \cite{Cin1972b}, which show that, roughly speaking, conditional on the path of $Y$, the process $X$ behaves like an additive process.
In our context, the modulating component $Y$ may be interpreted as describing changes in the economic environment, while the additive component $X$ represents the evolution of the firm's capital over time. 
The present paper studies the optimal dividend problem with capital injections for such MAP models. 
As will be explained in detail below, unlike previous studies in which the modulating component was restricted to a finite-state process, we allow $Y$ to be a general right process on a Radon space. 
This level of generality prevents a direct application of the proof strategies used in previous related work and 
leads us to develop a different approach.
Although our formulation is motivated by a particular dividend problem, the proof strategy developed here is not tied to this application alone.
One of the aims of this paper is to develop an argument for verifying optimality that may also be useful for a range of stochastic control problems involving general MAPs.
\par
There is already a substantial literature on this problem and closely related
problems when $E$ is finite (i.e., when $X$ is a L\'evy process or can be constructed by concatenating L\'evy processes). 
Since the literature in this direction is extensive, we discuss only those works whose settings are most closely related to the setting considered in the present paper. 
First, the work of \cite{AvrPalPis2007} played a crucial role in these studies. The authors of \cite{AvrPalPis2007} demonstrated the optimality of double barrier strategies in the classical--classical setting under the assumption that $X$ is a spectrally negative L\'evy process (a L\'evy process with no positive jumps and no monotone paths). 
Spectrally negative L\'evy processes are associated with scale functions.
Scale functions have been studied extensively, and their properties are now well understood (see, e.g., \cite{KuzKypRiv2012} or \cite[Chapter 8]{Kyp2014}). The authors of \cite{AvrPalPis2007} proved optimality by expressing the expected NPV under the candidate optimal strategy in terms of the scale function and showing that this expected NPV satisfies the Hamilton--Jacobi--Bellman (HJB) variational inequality. 
Subsequent studies employed this proof strategy for the classical--classical setting \cite{BayKypYam2013} and for the periodic--classical setting \cite{PerYam2017_2, NobPerYamYan2018} when $X$ is a spectrally one-sided L\'evy process. These studies have established the optimality of double barrier strategies and periodic--classical barrier strategies, respectively. 
Furthermore, \cite{Nob2021} and \cite{MatNobPer2025+} developed an alternative approach to the proof using the scale function for the classical--classical setting and the periodic--classical setting, respectively, when $X$ is a more general L\'evy process. 
The approach used in \cite{AvrPalPis2007} has also influenced subsequent work on a range of stochastic control problems driven by L\'evy processes, in which the optimality of explicit strategies is often verified via HJB variational inequalities.
Several prior studies also exist for the case in which $E$ consists of finitely many points.
Specifically, in the classical--classical setting, the case where $X$ is spectrally one-sided was studied in \cite{YanWan2026}, and the case where $X$ is more general was studied in \cite{MatNobPerYam2024}. In the periodic--classical setting, the spectrally one-sided case was studied in \cite{MatMorNobPer2023} and \cite{BoWanYan2025}. 
These studies followed the approach adopted in \cite{JiaPis2012}. 
\begin{enumerate}
\item By decomposing $X$ at the switching times of the constituent L\'evy processes, they reduced the problem to auxiliary stochastic control problems for L\'evy processes.
\item For these auxiliary problems, the optimality of double barrier strategies and periodic--classical barrier strategies was determined using an approach similar to that in \cite{AvrPalPis2007}.
\item Applying the results from (ii) to MAPs via dynamic programming (DP), they proved the optimality of Markov-modulated double barrier strategies and MMPCB strategies. 
\end{enumerate}
However, to the best of our knowledge, no prior work has addressed the setting considered here, in which $Y$ is a general right process. 
\par
The generality of the present MAP setting prevents us from using the standard proof strategy described above, which was developed for L\'evy and finite-state MAP models. 
Indeed, since we cannot construct $X$ by concatenating L\'evy processes at discrete switching times, as in prior works with finite $E$, we cannot reduce the problem to auxiliary stochastic control problems for L\'evy processes. 
More importantly, many of the HJB-based arguments in the preceding literature rely on the Meyer--It\^o formula or related generator calculations. 
Such arguments are not directly applicable in the present setting, because the modulating component $Y$ is allowed to be a general right process on a Radon space.
As a result, in this study, we first consider the periodic--classical setting and, broadly speaking, we adopt the following procedure to prove optimality: 
\begin{enumerate}
\item We demonstrate that a particular form of the capital injection component is necessary for optimality, thereby reducing the problem to one that considers only the dividend portion. 
\item We decompose the time horizon into intervals between dividend payment opportunities. Then, we use the DP to derive the necessary and sufficient conditions for a strategy to be optimal. 
Furthermore, these conditions also yield the optimality of certain MMPCB strategies. 
At this stage, the additivity property \eqref{200_01} is crucial. 
\item We characterize the conditions for optimality using the Laplace transform of hitting times. 
\end{enumerate}
This yields a result that includes, except for the representations in terms of scale functions, the results in \cite{PerYam2017_2, NobPerYamYan2018, MatMorNobPer2023, BoWanYan2025, MatNobPer2025+}.
Furthermore, by letting the intervals between dividend-payment opportunities shrink in the periodic--classical setting, we study the behavior of the barriers of the optimal MMPCB strategies and use this approximation to examine the possible 
barrier-type structure of an optimal strategy in the classical--classical setting. 
This also yields a result that includes the results in
\cite{BayKypYam2013, Nob2021, MatNobPerYam2024, YanWan2026}
and the results in the classical--classical setting of \cite{AvrPalPis2007},
except for the parts concerning representations in terms of scale functions. 
The method may also be useful beyond the present problem, for example in related optimal dividend problems (e.g., \cite{Loe2008, Nob2023}), non-dividend stochastic control problems of a similar form (e.g., \cite{Yam2017}), and L\'evy process control problems that admit a formulation within a general MAP framework.
\par
By adopting this approach, we can not only prove the optimality of specific strategies for general MAPs but also obtain several refinements of previous results. 
In contrast to prior studies, which focused on proving the optimality of specific strategies, we also provide, in step (ii) above, necessary and sufficient conditions for a strategy to be optimal.
In addition, previous studies for the case where $E$ is a finite set with at least two points did not provide an explicit expression for the optimal barrier of MMPCB strategies; rather, the optimal barrier was characterized only implicitly through approximations used in the proof.
In contrast, the present paper shows that, in the MAP setting, the optimal barrier can be characterized in terms of the Laplace transforms of suitable hitting times, in the same spirit as in the L\'evy process case, thereby extending the previous results. 
Furthermore, prior work on general L\'evy processes such as \cite{Nob2021, MatNobPerYam2024, MatNobPer2025+} 
relied on smoothness assumptions needed 
to apply the Meyer--It\^o formula. In this paper, however, we avoid using that formula and thus eliminate the need for such assumptions. 
 \par
The structure of this paper is as follows. In Section \ref{Sec002}, we introduce the MAP considered in this paper, and formulate the optimal dividend problems in the classical--classical and periodic--classical settings. 
We also define MMPCB strategies and related stochastic processes. 
In Section \ref{Optimality_pcb}, we discuss the optimal dividend problem in the periodic--classical setting and provide the necessary and sufficient conditions for a strategy to be optimal.
We also demonstrate the optimality of certain MMPCB strategies. 
Section \ref{Sec_approximation} examines two approximations of the classical--classical setting by periodic--classical settings and considers the possible form of optimal strategies in the classical--classical setting.
The Supplementary Material collects the proofs of several lemmas, propositions, and theorems whose arguments are either straightforward, similar to those given elsewhere in the paper, or rather technical and not essential for understanding the main flow of the paper.

\section{Preliminaries} \label{Sec002}
\subsection{Markov additive process}
Let $(E, \cB(E))$ be a Radon space (for the definition of Radon spaces, see, e.g., \cite[Appendix A2]{Sha1988}). 
We assume that 
 $(X, Y)=(\Omega, \cF, \cF_t, (X_t, Y_t), \theta_t, \bfP_{(x, y)})$ 
is an $\bR\times E$-valued right process without killing, and that $X$ has c\`adl\`ag paths (for the definition of right processes, see, e.g., \cite[p.38]{Sha1988}). 
Here, $\{X_t:t\geq0\}$ takes values in $\bR$ and $\{Y_t:t\geq0\}$ takes values in $E$. 
We assume that $(X, Y)$ satisfies \eqref{200_01} for any non-negative measurable function $f$ on $\bR\times E$ and $t\geq 0$, where $\bfE_{(x, y)}$ denotes the expectation under $\bfP_{(x, y)}$ for $(x, y)\in\bR\times E$. 
Then, $(X, Y)$ is called a Markov additive process (MAP). 
Note that $Y=(\Omega, \cF, \cF_t,  Y_t, \theta_t, \bfP^{Y,(x)}_{y})$, where, for fixed $x \in \bR$, $\bfP^{Y, (x)}_y:=\bfP_{(x,y)}$, is also a right process. The law of this process does not depend on the choice of $x\in\mathbb{R}$, 
so when we focus only on $Y$, we omit the superscript $(x)$ in $\bfP^{Y,(x)}_{y}$ and simply write $\bfP^{Y}_{y}$. 
We write $\bfE^Y_y$ for the expectation with respect to $\bfP^Y_y$.
\par
Throughout this paper, we impose the following assumption.
\begin{align}
M:=\sup_{y\in E} \bfE_{(0, y)}\sbra{\sup_{s \in [0, 1]}|X_s| }<\infty. 
\label{finite_ass}
\end{align}
We use this assumption to show the admissibility of certain strategies (see Lemma \ref{Lem304}). 
In the case of L\'evy processes, it is straightforward to verify that this is equivalent to \cite[Assumption 2.1]{Nob2021} (see also \cite[Assumption 2.1]{MatNobPer2025+}), using an argument similar to the latter part of the proof of \cite[Lemma 3.2]{Nob2021}.
\par
We adopt the following conventions. Any term involving a discount factor corresponding to $e^{-\infty}$ is understood to be zero, even if the factor multiplied by this discount factor is not defined on the event under consideration. Moreover, expressions of the form $\infty+\cdot\circ\theta_\infty$ are interpreted as $\infty$.

\subsection{Optimal dividend problem with capital injections}
\label{SubsubSec221}
We fix a measurable function $\bfq$ from $E$ to $(q, \infty)$ with $q>0$ and a constant $\beta>1$. 
We denote $\frq(t)= \int_0^t\bfq(Y_s) \diff s$ for $t\geq 0$. 

\subsubsection{Classical--classical setting}\label{CCcase}
We define a strategy as any $\bR^2$-valued process $\pi=\{ (L^\pi_t,R^\pi_t): t\geq 0 \}$ 
that satisfies the following conditions.  
\begin{enumerate}
\item The maps $t\mapsto L^\pi_t$ and $t\mapsto R^\pi_t$ are non-negative, non-decreasing and c\`adl\`ag. 
\item For $t\geq 0$, $L^\pi_t$ and $R^\pi_t$ are $\cF_t$-measurable. 
\item $U^\pi_t \geq 0$ for $t\geq 0$, where $U^\pi_t := X_t-L^\pi_t+ R^\pi_t$ for all $(x,y)\in \bR \times E$, $\bfP_{(x, y)}$-a.s. 
\end{enumerate}
In particular, the strategy $\pi$ satisfying 
\begin{align}
\bfE_{(x, y)}\sbra{\int_{[0, \infty)}e^{-\frq(t)} \diff R^\pi_t}<\infty,
\qquad (x, y)\in\bR\times E, 
\label{adcon}
\end{align} 
is called an admissible strategy. 
We denote by $\Pi$ the set of all admissible strategies. 
For a strategy $\pi\in\Pi$, the expected NPV of dividends and capital injections of $\pi$ is defined, for $ (x, y )\in\bR\times E$, by
\begin{align}
v_\pi(x, y):=v^L_\pi(x, y)-\beta v^R_\pi(x, y), \label{NPV}
\end{align}
where 
$v^L_\pi(x, y):=\bfE_{(x, y)}\sbra{\int_{[0,\infty)}e^{-\frq(t)} \diff L^\pi_t}$, and 
$v^R_\pi(x, y):=\bfE_{(x, y)}\sbra{\int_{[0,\infty)}e^{-\frq(t)} \diff R^\pi_t}$. 
Our objective is to identify a strategy $\pi^{\ast}\in\Pi$ that satisfies
\begin{align}
v_{\pi^{\ast}}(x, y)=V(x, y):=\sup_{\pi\in\Pi} v_\pi(x, y),
\quad (x, y )\in\bR\times E. \label{opccb}
\end{align}
The function $V$ is called the value function of $\Pi$ and 
a strategy $\pi^{\ast}$ satisfying \eqref{opccb} is called an optimal strategy in $\Pi$. 
In this paper, we use this setting only to infer the possible form of optimal strategies through approximation.

\subsubsection{Periodic--classical setting} \label{PCcase}
Let $\nu$ be a probability distribution on $(0, \infty)$. 
Let $\bfr:=\{r_k: k \in \bN\}$ be a sequence of random variables defined on $(\Omega, \cF)$, which are independent and identically distributed with common distribution $\nu$, and are independent of $\{(X_t, Y_t):t\geq 0\}$ under each $\bfP_{(x, y)}$ with $(x, y) \in \bR \times E$. 

We define the process $N^\nu:=\{N^\nu_t : t\geq 0\}$ as
$N^\nu_t := \max\cbra{k \in \bN: \sum_{m=1}^kr_m\leq t}$ for $t\geq 0$, where $\max \varnothing =0$. 
Since the common distribution of the $r_k$'s is fixed as $\nu $, we will simply use $\bfP$ and $\bfE$ when considering probabilities or expectations involving only $\bfr$.
\par
We shall assume that $N^\nu_t$ is $\cF_t$-measurable for $t\geq 0$, $T^\nu_n$ is a stopping time and $\{N^\nu_{t+T^\nu_n}-N^\nu_{T^\nu_n}:t\geq 0\}$ is independent of $\cF_{T^\nu_n}$ for $n\in\bN$, where $T^\nu_n:=\sum_{k=1}^n r_k$. 
Indeed, separately from the probability space $(\Omega,\cF,\bfP_{(x,y)})$ on which $(X,Y)$ is defined, 
we may consider another probability space $(\Omega^\prime,\cF^\prime,\bfP^\prime)$ on which 
$\bfr$ is defined, and then take the product space. 
We may then relabel $\Omega\times\Omega^\prime$ as $\Omega$, the universal completion of $\cF\otimes\cF^\prime$ as $\cF$, 
$(\omega, \omega^\prime)\mapsto (\theta_t \omega, \omega^\prime)$ as $\theta_t$ for $t\geq 0$ and $(\omega,\omega^\prime)\in\Omega\times\Omega^\prime$,
and $\bfP_{(x,y)}\otimes\bfP^\prime$ as $\bfP_{(x,y)}$ for $(x, y)\in \bR\times E$. 
Moreover, we denote by $\cF^\prime_t$ the $\sigma$-field generated by $\cF_t$, $\{N^\nu_s:s\in[0,t]\}$ and $\cN (\cF\otimes\cF^\prime)$ (for the definition of $\cN$, see \cite[p.19]{Sha1988}) and 
relabel $\cF^\prime_{t+}$ as $\cF_t$ for $t\geq 0$. 
\par
We emphasize here that a property analogous to the Markov property holds: for $(x, y)\in\bR\times E$, $ n \in \bN $, $A\in\cF_{T^\nu_n}$ and for all non-negative measurable functions $ f $, we have 
\begin{align}
\bfE_{(x, y)}&\sbra{f(\{(X_{t+T^\nu_n}, Y_{t+T^\nu_n}, N^\nu_{t+T^\nu_n}-n):t\geq 0\}
)
; A}\\
&\qquad =\bfE_{(x, y)}\sbra{\bfE_{(X_{T^\nu_n}, Y_{T^\nu_n})}\sbra{ f(\{(X_t, Y_t,N^\nu_t):t\geq 0\})};A},
\label{periodicMarkov}
\end{align}
Hereafter, we write $T^{\nu}$ for $T^{\nu}_1(=r_1)$, and set $T^{\nu}_0=0$ and $T^{\nu}_\infty=\infty$.
\par
In the periodic--classical setting, a strategy is any $\bR^2$-valued process $\pi=\{ (L^{\pi}_t,R^{\pi}_t): t\geq 0 \}$ that satisfies conditions (i)--(iii) in Section \ref{SubsubSec221} 
and the following additional condition. 
\begin{itemize}
\item[(iv)] There exists a non-negative progressively measurable process $\ell^{\pi}:=\{\ell^{\pi}_t : t\geq 0\}$ such that 
$L^{\pi}_t = \int_{[0,t]}\ell^{\pi}_s \diff N^\nu_s$ for $ t\geq 0$,
for all $(x,y)\in \bR \times E$, $\bfP_{(x, y)}$-a.s.
\end{itemize} 
As in the classical--classical setting, 
a strategy $\pi$ satisfying \eqref{adcon}
is called an admissible strategy. 
We denote by $\Pi^\nu$ the set of all admissible strategies. 
For a strategy $\pi\in\Pi^\nu$, the expected NPV of dividends and capital injections of $\pi$ is defined by \eqref{NPV}. 
As in the classical--classical setting, 
our objective is to identify a strategy $\pi^{\nu,\ast}\in\Pi^\nu$ that satisfies
\begin{align}
v_{\pi^{\nu,\ast}}(x, y)=V^{\nu}(x, y):=\sup_{\pi\in\Pi^\nu} v_\pi(x, y),
\quad (x, y )\in\bR\times E. \label{oppcb}
\end{align}
The function $V^{\nu}$ is called the value function of $\Pi^\nu$ and 
a strategy $\pi^{\nu,\ast}$ satisfying \eqref{oppcb} is called an optimal strategy in $\Pi^\nu$.

\subsection{Markov-modulated periodic reflection}
\label{SubSec203}

In this section, we define processes and strategies constructed by Markov-modulated periodic reflection. 
In the context of the optimal dividend problem in the periodic--classical setting,
we will show that these strategies are indeed optimal. 
\par
We consider the problem under the same framework as in Section \ref{PCcase}. 
Let $\bfb$ be a measurable function on $E$. 
We define the Markov-modulated periodically reflected Markov additive process as 
\begin{align}
Z^\bfb_t := X_t - \rbra{\max_{T^\nu_k\in (0, t]} (X_{T^\nu_k}-\bfb(Y_{T^\nu_k}))\lor 0} , 
\qquad t\geq 0, 
\end{align}
where $\max\varnothing = 0$. 
We also define the Markov-modulated periodic--classical barrier (MMPCB) strategy $\pi^\nu_\bfb$ with barrier $\bfb$ by inductively applying the following operation ($\star$). 
\par
Set $\eta=0$, $\chi=\bfb(Y_0)$ and $L^{\pi^\nu_\bfb}_{0-}=R^{\pi^\nu_\bfb}_{0-}=0$. 
\begin{itemize}
\item[($\star$)] We define the process $\{Z_t : t\geq 0\}$ as 
\begin{align}
Z_t := 
\begin{cases}
X_t ,\qquad &\text{ if }\eta =0,
\\
\bfb(Y_\eta)+X_t -X_\eta, &\text{ if } \eta>0, 
\end{cases}
\end{align}  
for $t\geq 0$. 
We define the process $Z^0:=\{Z^0_t : t\geq \eta\}$ as 
$Z^0_t:= Z_t - \rbra{\inf_{s\in[\eta, t]} Z_s\land 0}$ for $t\geq \eta$. 
We also set $\eta^\prime:=\min\{T^\nu_k : T^\nu_k>\eta, Z^0_{T^\nu_k}>\bfb(Y_{T^\nu_k})\}$, where $\min\varnothing=\infty$. 
We set, for $t \in[\eta, \eta^\prime)$, 
\begin{align}
L^{\pi^\nu_\bfb}_t = L^{\pi^\nu_\bfb}_{\eta-}+ (\chi-\bfb(Y_\eta))\lor0
,\quad R^{\pi^\nu_\bfb}_t =R^{\pi^\nu_\bfb}_{\eta-}
- \rbra{\inf_{s\in[\eta, t]} Z_s\land0}.
\end{align}
If $\eta^\prime=\infty$, then the construction stops here. 
If $\eta^\prime<\infty$, we reset $\eta$ to $\eta^\prime$ and $\chi$ to $Z^0_{\eta^\prime}$, and then proceed to the next step. 
\end{itemize}
We denote the resulting controlled process by $U^{\pi^\nu_\bfb}:=\{U^{\pi^\nu_\bfb}_t: t\geq 0\}$. 
We denote by $\Pi^\nu_\pcb$ the set of MMPCB strategies $\pi^\nu_\bfb$ such that $\bfb$ is non-negative.

\section{Optimal strategies in the periodic--classical setting}\label{Optimality_pcb}
In this section, we consider the optimal dividend problem in the setting described in Section \ref{PCcase}. We fix the distribution $\nu$ and find optimal strategies in $\Pi^\nu$ satisfying \eqref{oppcb}. 
We will first summarize the main results of this section below.  
\par
We define the hitting times $\tau^-_0$ and $\tau^-_{0+}$ of $(X, Y)$ as  
\begin{align}
\tau^-_0:= \inf\{t>0 : X_t < 0\},
\qquad\tau^-_{0+}:= \inf\{t\geq 0 : X_t \leq 0\},
\end{align} 
where $\inf\varnothing=\infty$. 
For a measurable function $\bfb$, we define hitting times as  
\begin{align}
T^{\nu,+}_\bfb:=&\min\{T^\nu_k>0: X_{T^\nu_k}>\bfb{(Y_{T^\nu_k})} \}, 
\qquad 
T^{\nu,+}_{\bfb-}:=\min\{T^\nu_k>0: X_{T^\nu_k}\geq \bfb{(Y_{T^\nu_k})} \}, 
\end{align}
where $\min\varnothing =\infty$. 
We define, for a non-negative measurable function $\bfb$ on $E$ and $y\in E$, 
\begin{align}
\rho^{\nu,1}_{\bfb}(y)&:=\bfE_{(\bfb(y),y)} \sbra{ e^{-\frq(T^{\nu,+}_{\bfb-})}; T^{\nu,+}_{\bfb-}\leq\tau^-_0}+\beta \bfE_{(\bfb(y),y)} \sbra{ e^{-\frq(\tau^-_0)}; \tau^-_0<T^{\nu,+}_{\bfb-}}, 
\\
\rho^{\nu,2}_{\bfb}(y)&:=\bfE_{(\bfb(y),y)} \sbra{ e^{-\frq(T^{\nu,+}_{\bfb})}; T^{\nu,+}_{\bfb}<\tau^-_{0+}}+\beta \bfE_{(\bfb(y),y)} \sbra{ e^{-\frq(\tau^-_{0+})}; \tau^-_{0+}< T^{\nu,+}_{\bfb}}, 
\end{align}
which are measurable functions on $E$. 
We define $\Xi^\nu$ as the set of all non-negative measurable functions $\bfb$ on $E$ satisfying $E=\hat{E}^\bfb$ where
$\hat{E}^\bfb:=\{y\in E: \rho^{\nu,1}_{\bfb}( y)\leq1,\ \rho^{\nu,2}_{\bfb}( y)\geq 1\}$.
$\Xi^\nu$ is used to characterize the barriers of optimal MMPCB strategies by inequalities. In previous studies, the values of optimal barriers have been expressed using Laplace transforms of various hitting times (e.g., \cite[p.81]{Nob2021}, \cite[p.181]{Nob2023} and \cite[(4.13)]{MatNobPer2025+}). $\Xi^\nu$ represents a generalization to the MAP case. The rationale is as follows. For certain L\'evy processes, such as spectrally negative L\'evy processes and meromorphic L\'evy processes, two-sided exit formulae are available and provide explicit expressions for the Laplace transforms of particular hitting times under suitable conditions; see, e.g., \cite[Chapter 8]{Kyp2014} and \cite{KuzKypPar2012}. These formulae can then be used to derive explicit expressions for the Laplace transforms of various hitting times. For MAPs, however, the explicit computation of such Laplace transforms remains a topic for future research. 
\begin{Rem}
In light of previous studies, including \cite[(4.13)]{MatNobPer2025+}, it seems reasonable to expect that the set of optimal barriers $\breve{\Xi}^\nu$ is the set of non-negative measurable functions $\bfb$ on $E$ satisfying
$\bfE_{(\bfb(y),y)} \sbra{ e^{-\frq(\kappa^{\nu,\bfb}_0)}}\leq1$ and 
$\bfE_{(\bfb(y),y)} \sbra{ e^{-\frq(\kappa^{\nu,\bfb}_{0+})}}\geq 1$ 
for $y\in E$,
where
\begin{align}
\kappa^{\nu,\bfb}_0:=\min\{t>0: Z^{\bfb}_t<0 \}, 
\qquad 
\kappa^{\nu,\bfb}_{0+}:=\min\{t\geq 0: Z^{\bfb}_t\leq 0 \}, 
\end{align}
with $\min\varnothing=\infty$.
In fact, $\Xi^\nu\subset\breve{\Xi}^\nu$ follows from Lemma \ref{LemB01}. However, apart from special cases, $\Xi^\nu\supset\breve{\Xi}^\nu$ is not obvious, and it is not verified in this paper.
\end{Rem}
$\Xi^\nu$ is characterized as follows. 
\begin{Thm}\label{Thm302a}
There exist unique finite-valued measurable functions $\un{\bfb}^\nu$ and $\bar{\bfb}^\nu$ on $E$ such that $\un{\bfb}^\nu,\bar{\bfb}^\nu\in \Xi^\nu$, 
and for every non-negative measurable function $\bfb$ on $E$, $\bfb$ belongs to $\Xi^\nu$ if and only if $\un{\bfb}^\nu( y)\leq\bfb(y)\leq\bar{\bfb}^\nu(y)$ for all $y\in E$. 
\end{Thm}
Let $\Pi^{\nu, \ast}$ be the set of strategies $\pi$ in $\Pi^\nu$ for which 
the following conditions hold $\bfP_{(x,y)}$-a.s. for all $(x, y)\in\bR \times E$. 
\begin{enumerate}
\item At each time $T^\nu_n$ with $n\in\bN$, the dividend payout is made as follows. 
\begin{enumerate} 
\item If $U^\pi_{(n-)}\leq\un{\bfb}^\nu(Y_{T^\nu_n})$, then no dividend is paid, where for $n\in\bN$, $U^\pi_{(n-)}:=U^\pi_{T^\nu_n-} + \Delta X_{T^\nu_n}$ denotes the value of the controlled process at time $T^\nu_n$ when the strategy $\pi$ is followed up to, but not including, $T^\nu_n$, and no payout is made at $T^\nu_n$. 
\item If $U^\pi_{(n-)}>\un{\bfb}^\nu(Y_{T^\nu_n})$, then dividends may be paid and $U^\pi_{T^\nu_n}$ belongs to $ [\un{\bfb}^\nu(Y_{T^\nu_n}), \bar{\bfb}^\nu(Y_{T^\nu_n})]$. 
\end{enumerate}
\item $\{R^\pi_t: t\geq 0\}$ satisfies 
$R^\pi_t= - \rbra{\inf_{s \in[0, t]}(X_s-L^\pi_s) \land 0}$ for $ t\geq 0$. 
\end{enumerate}
Here, we denote $\Delta Z_t:=Z_t-Z_{t-}$ for any process $\{Z_t: t\geq 0\}$ and $t\geq 0$. 
\begin{Thm}\label{newThm302}
The strategy $\pi\in\Pi^\nu$  is optimal if and only if it belongs to $\Pi^{\nu,\ast}$.
\end{Thm}
Let $\hat{\Xi}^\nu$ be the set of non-negative measurable functions $\bfb$ on $E$ satisfying $m_y (E\backslash E^\bfb)=0$ for all $y\in E$, where 
$E^\bfb:=\cbra{y \in E: \un{\bfb}^\nu( y)\leq\bfb(y)\leq\bar{\bfb}^\nu(y)}$, and
$m_y (B ):= \bfP^{Y}_y (Y_{T^\nu}\in B)$ for $B\in\cB(E)$. 
Note that $\hat{\Xi}^\nu$ can be characterized as follows.
\begin{Prop}\label{Prop303}
A measurable function $\bfb$ satisfies $\bfb\in\hat{\Xi}^\nu$ if and only if 
$m_y (E\backslash \hat{E}^\bfb)=0$ for all $y \in E$. 
\end{Prop}
For the proof of Proposition \ref{Prop303}, see Appendix \ref{Prf_Prop303} of the Supplementary Material. 
For $\pi^\nu_\bfb \in\Pi^\nu_\pcb$, it is immediately clear that $\bfb\in\hat{\Xi}^\nu$ is equivalent to $\pi^\nu_\bfb \in\Pi^{\nu,\ast}$. 
Therefore, by Theorem \ref{newThm302}, the following corollary follows. 
\begin{Cor}\label{Thm302}
For any non-negative measurable function $\bfb$ on $E$, the MMPCB strategy $\pi^\nu_\bfb$ is optimal if and only if $\bfb\in\hat{\Xi}^\nu$. 
\end{Cor}
Furthermore, the following theorem characterizes $\un{\bfb}^\nu$ and $\bar{\bfb}^\nu$ in a special case.
\begin{Thm}\label{Thm303}
When $\nu$ has an exponential distribution with intensity $r>0$, the functions $\un{\bfb}^\nu$ and $\bar{\bfb}^\nu$ are finely lower semi-continuous and finely upper semi-continuous, respectively, with respect to the right process $Y$. 
\end{Thm}
One of the main features of right processes is the following property: for every $\alpha>0$ and every $\alpha$-excessive function $f$, the process $t\mapsto f(X_t, Y_t)$ is right-continuous a.s. This property is used mainly in the proof of Theorem \ref{Thm303}. In the proofs of the other results, it seems possible to impose weaker assumptions on $(X,Y)$. 
\par
In this section, we will prove the above theorems. To that end, several lemmas will be required. 
We will proceed with the proofs of Theorems \ref{Thm302a} and \ref{newThm302} and Corollary \ref{Thm302} in the following steps.
\begin{itemize}
\item[\textbf{Step 1}] We define $\un{\Pi}^\nu \subset\Pi^\nu$ as a more tractable subclass of strategies that are easier to handle, and show that the problem can be reformulated as one of finding the necessary and sufficient conditions for a strategy in $\un{\Pi}^\nu$ to be optimal. 
\item[\textbf{Step 2}] 
Using the DP, we examine the necessary and sufficient conditions for a strategy belonging to $\un{\Pi}^\nu$ to be optimal, and prove 
a preliminary version of Theorem \ref{newThm302}, in which $\un{\bfb}^\nu( y)$ and $\bar{\bfb}^\nu(y)$ are replaced by barriers in the definition of $\Xi^\nu_{V^\nu}$, which is defined later.  
\item[\textbf{Step 3}] We show that $\Xi^\nu$ and $\Xi^\nu_{V^\nu}$ are equal. 
\end{itemize}
\par
\textbf{Step 1.} Let $\un{\Pi}^\nu$ be the set of strategies $\pi$ that satisfy
\begin{align}
\Delta L^\pi_t \leq (U^\pi_{t-}+\Delta X_t)\lor 0 ,\qquad 
R^\pi_t = -\rbra{\inf_{s \in[0, t]}(X_s-L^\pi_s) \land 0},\qquad t\geq 0, 
\label{300_01}
\end{align}
$\bfP_{(x, y)}$-a.s. for any $(x, y)\in\bR\times E$. 
Note that $\Pi^\nu_\pcb \subset\un{\Pi}^\nu$ and $\Pi^{\nu,\ast}\subset\un{\Pi}^\nu$. 
In addition, since the capital injections for strategies in $\un{\Pi}^\nu$ are uniquely determined by their dividend components, it is sufficient to describe only the dividend components when we define the strategies in $\un{\Pi}^\nu$. 
Therefore, in what follows, we will often describe only the dividend component when defining strategies belonging to $\un{\Pi}^\nu$.
\par
Let $L^0=\{L^0_t:t\geq 0\}$ and $R^0=\{R^0_t:t\geq 0\}$ denote the dividend and capital injection processes of the following strategy. At time $0$, pay dividends $x\lor 0 $ and inject capital $-(x\land 0)$. Thereafter, the strategy follows the MMPCB strategy with barrier $ 0$. 
The following lemma implies that $\un{\Pi}^\nu \subset \Pi^\nu$. 
\begin{Lem}\label{Lem304}
For all $\pi\in\un{\Pi}^\nu$ and $(x,y)\in\bR\times E$, we have 
\begin{align}
\int_{[0,\infty)}e^{-\frq(t)} \diff L^\pi_t \leq \int_{[0,\infty)}e^{-\frq(t)} \diff L^0_t,\qquad
\int_{[0,\infty)}e^{-\frq(t)} \diff R^\pi_t \leq \int_{[0,\infty)}e^{-\frq(t)} \diff R^0_t, 
\label{upperbound_by_0strategy}
\end{align}
$\bfP_{(x, y)}$-a.s.
In addition, there exists $B^\nu>0$ such that for $\pi\in\un{\Pi}^\nu$ and 
$(x, y)\in\bR\times E$, 
\begin{align}
\bfE_{(x, y)}\sbra{\int_{[0,\infty)}e^{-\frq(t)} \diff L^0_t}
\leq (x\lor0)+B^\nu,\qquad
\bfE_{(x, y)}\sbra{\int_{[0,\infty)}e^{-\frq(t)} \diff R^0_t}
\leq  -(x\land 0)+ B^\nu.
\label{finiteness}
\end{align}
\end{Lem}
Since roughly half of the proof is essentially the same as that of \cite[Lemma 3.2]{Nob2021}, we give the proof in Appendix \ref{lemma_proof} of the Supplementary Material.
\par
Together with Lemma \ref{Lem304}, the following lemma shows that, to establish one of our main results, it suffices to study necessary and sufficient conditions for a strategy in $\un{\Pi}^\nu$ to be optimal. 
\begin{Lem}\label{Lem304a}
We have 
\begin{align}
\sup_{\pi\in\un{\Pi}^\nu} v_\pi(x, y)
=V^{\nu} (x, y),
\quad (x, y )\in\bR\times E. \label{Lem304a_001}
\end{align}
In addition, any $\pi\in \Pi^\nu\backslash\un{\Pi}^\nu$ satisfies 
$v_\pi(x, y) < V^{\nu} (x, y)$, 
for some $(x, y)\in\bR\times E$. 
\end{Lem}
\begin{proof}
For \eqref{Lem304a_001}, it is sufficient to show that for a given $\pi\in\Pi^\nu$, there exists $\un{\pi}\in\un{\Pi}^\nu$ such that 
\begin{align}
v_{\un{\pi}}(x, y)
\geq  v_\pi(x, y),
\quad (x, y )\in\bR\times E. \label{Lem305_001a}
\end{align}
For fixed $\pi\in\Pi^\nu$, we define $\pi^\prime\in\Pi^\nu$ as 
\begin{align}
L^{\pi^\prime}_t=L^{\pi}_t,\quad 
R^{\pi^\prime}_t=-\rbra{\inf_{s \in[0, t]}(X_s-L^{\pi}_s)\land 0},\qquad t\geq 0. 
\label{ubs}
\end{align} 
Then, the strategy $\pi^\prime$ satisfies the second equation of \eqref{300_01}. 
From condition (iii) in the definition of strategies in $\Pi^\nu$ and \eqref{ubs}, we have, for $t\geq 0$,  
\begin{align}
R^{\pi}_t-R^{\pi^\prime}_t=R^{\pi}_t+\inf_{s \in[0, t]}(X_s-L^{\pi}_s)
\geq\inf_{s \in[0, t]}(X_s-L^{\pi}_s+R^{\pi}_s)\geq 0.
\end{align} 
Thus, we have, for $(x, y)\in\bR \times E$,  
\begin{align}
v_{\pi^\prime}(x, y)-  v_{\pi}(x, y)
=&\beta\bfE_{(x, y)}\sbra{\int_{[0,\infty)}e^{-\frq(t)} \diff R^{\pi}_t-\int_{[0,\infty)}e^{-\frq(t)} \diff R^{\pi^\prime}_t}\\
=&\beta\bfE_{(x, y)}\sbra{\int_0^\infty \bfq(Y_s) e^{-\frq(s)} ( R^{\pi}_s-R^{\pi^\prime}_s)\diff s}\geq 0. \label{ineq_002}
\end{align}
We define $\un{\pi}\in\Pi^\nu$ as 
\begin{align}
L^{\un{\pi}}_t=\sum_{s \in[0, t]}\Delta L^{\pi^\prime}_s\land \cbra{(U^{\pi^\prime}_{s-}+\Delta X_s)\lor 0},\quad 
R^{\un{\pi}}_t=R^{\pi^\prime}_t -\tilde{L}^{\pi^\prime,\un{\pi}}_t,\qquad t\geq 0, 
\end{align} 
where $\tilde{L}^{\pi^\prime,\un{\pi}}_t:=\sum_{s\in[0,t]}(\Delta L^{\pi^\prime}_s-\Delta L^{\un{\pi}}_s)$ for $t\geq 0$.
Then, the resulting controlled processes corresponding to $\pi^\prime$ and $\un{\pi}$ coincide, and it is easy to verify that $\un{\pi}\in \un{\Pi}^\nu$.
By the definition of $\un{\pi}$, we have, for $(x, y)\in\bR \times E$,  
\begin{align}
&v_{\un{\pi}}(x, y)-v_{\pi^\prime}(x, y)\\
&\ \ =\bfE_{(x, y)}\bigg{[}\int_{[0,\infty)}e^{-\frq(t)} \diff L^{\un{\pi}}_t-\int_{[0,\infty)}e^{-\frq(t)} \diff L^{\pi^\prime}_t 
-\beta\rbra{\int_{[0,\infty)}e^{-\frq(t)} \diff R^{\un{\pi}}_t-\int_{[0,\infty)}e^{-\frq(t)} \diff R^{\pi^\prime}_t}\bigg{]}\\
&\ \ =\bfE_{(x, y)}\sbra{(\beta-1)\int_{[0,\infty)}e^{-\frq(t)} \diff \tilde{L}^{\pi^\prime,\un{\pi}}_t}\geq 0, \label{ineq_001}
\end{align}
where the last inequality follows from the fact that the process $\cbra{ \tilde{L}^{\pi^\prime,\un{\pi}}_t:t\geq 0}$ is non-decreasing. 
By \eqref{ineq_002} and \eqref{ineq_001}, we have \eqref{Lem305_001a}. 
\par 
If $\pi\in\Pi^\nu$ does not belong to $\un{\Pi}^\nu$, 
there exists $(x, y) \in \bR \times E$ such that the event that \eqref{300_01} fails has positive $\bfP_{(x, y)}$-probability. Let us fix such a $(x, y)$.
First, we assume that the second equation in \eqref{300_01} fails with positive probability under the probability measure $\bfP_{(x, y)}$. 
Then, with positive $\bfP_{(x,y)}$-probability, the set of times at which $R^{\pi}_t-R^{\pi^\prime}_t=R^{\pi}_t+\inf_{s \in[0, t]}(X_s-L^{\pi^\prime}_s)>0$ has strictly positive Lebesgue measure, and hence 
the last term in \eqref{ineq_002} is strictly positive. 
Thus, $\pi^\prime$ strictly improves upon $\pi$, and $\pi$ is not an optimal strategy. 
Next, we assume that the second equation in \eqref{300_01} always holds, whereas the first inequality in \eqref{300_01} fails with positive probability under the probability measure $\bfP_{(x, y)}$.
In this case, $\pi$ and $\pi^\prime$ coincide. 
However, since the last term in \eqref{ineq_001} is strictly positive, $\un{\pi}$ yields a strictly larger value than $\pi$, and $\pi$ is not an optimal strategy. 
The proof is complete. 
\end{proof}
\par
\textbf{Step 2.} 
As mentioned earlier, we apply the DP here. To do so, we define a class of functions related to the expected NPV of dividends and capital injections, and then define operators that act on the functions in that class. 
\par
For a measurable function $f$ on $\bR \times E$ such that $x\mapsto f(x, y)$ has a right derivative for $y\in E$, let $x\mapsto f^\prime_+(x, y)$ denote its right derivative.  
Similarly, if $x\mapsto f(x,y)$ has a left derivative for $y\in E$, then we let $x\mapsto f^\prime_-(x,y)$ denote its left derivative.
Let $\Gamma^\nu$ be the set of measurable functions $f$ on $\bR \times E$ that satisfy the following conditions. 
\begin{enumerate}
\item For $y \in E$, $x\mapsto f(x, y)$ is a concave function. 
\item For $y \in E$, the right derivative $x\mapsto f^\prime_+(x, y)$ 
satisfies 
$f^\prime_+(x, y)=\beta$ for $x <0$ and    
$f^\prime_+(x, y) \in [0, \beta]$ for $x\geq 0$. 
\item For $(x,y)\in [0,\infty)\times E$, the following inequalities hold
\begin{align}
f(x, y)\leq \bfE_{(x,y)}\sbra{\int_{[0,\infty)} e^{-\frq(t)} \diff L^0_t},\qquad  f(0, y)\geq  -\beta \bfE_{(0,y)}\sbra{\int_{[0,\infty)} e^{-\frq(t)} \diff R^0_t}. 
\end{align}
\end{enumerate} 

To formulate the DP, 
we define operators $\cV^{\nu}_\pi$ with $\pi \in \un{\Pi}^\nu$ and $\cV^{\nu}$ acting on $\Gamma^\nu$ as follows: for $f \in \Gamma^\nu$ and $(x, y) \in \bR\times E$, 
\begin{align}
\cV^{\nu}_\pi f (x, y)= 
\bfE_{(x, y)}\sbra{ e^{-\frq(T^\nu)} L^\pi_{T^\nu}  -\beta\int_{[0,T^\nu]}e^{-\frq(t)} \diff R^\pi_t+ e^{-\frq(T^\nu)} f(U^\pi_{T^\nu}, Y_{T^\nu})}
, 
\end{align}
and
$\cV^{\nu} f (x, y)=\sup_{\pi\in\un{\Pi}^\nu} \cV^{\nu}_\pi f (x, y)$. 
Although the operators $\cV^{\nu}_\pi$ and $\cV^\nu$ are primarily applied to functions in $\Gamma^\nu$, we shall occasionally apply them to other functions on $\bR\times E$, provided that the expression is well-defined.
\par
For $f \in\Gamma^\nu$, we define non-negative measurable functions $\un{\bfb}_f$ and $\bar{\bfb}_f$ on $E$ as 
\begin{align}
\un{\bfb}_f (y):=& \sup \{x \geq 0: f^\prime_+(x, y)>1\},
\quad
\bar{\bfb}_f (y):= \inf \{x \geq 0 : f^\prime_+(x, y)<1\},
\qquad y \in E, 
\end{align}
where $\sup \varnothing =0$ and $\inf\varnothing=\infty$. 
Note that $\un{\bfb}_f (y)\leq \bar{\bfb}_f (y)$ for $y \in E$. 
We can rewrite $\cV^{\nu}$ in terms of a concrete strategy in $\un{\Pi}^\nu$ as follows. 
\begin{Lem}\label{Lem303}
We fix $f \in \Gamma^\nu$. 
Let $\bfb_f$ be a measurable function on $E$ such that 
$\un{\bfb}_f (y)\leq\bfb_f(y) \leq \bar{\bfb}_f (y)$ for $m_{y^\prime}$-a.e. $y \in E$ for all ${y^\prime}\in E$. 
Then, for $(x, y )\in\bR\times E$, we have 
$\cV^{\nu}f(x, y)=\cV^{\nu}_{\pi^\nu_{\bfb_f}}f(x, y)$. 
\end{Lem}
\begin{proof}
By \eqref{300_01}, we have, for $\pi\in\un{\Pi}^\nu$, 
\begin{align}
\beta\int_{[0,T^\nu]}e^{-\frq(t)} \diff R^{\pi}_t  
=\beta\int_{[0,T^\nu]}e^{-\frq(t)} \diff ((-\un{X}_t)\lor 0) ,
\label{303_02}
\end{align}
for all $(x,y)\in\bR\times E$, $\bfP_{(x, y)}$-a.s., where $\un{X}_t:= \inf_{s \in[0, t]} X_s$ for $t\geq 0$.  
Moreover, 
\begin{align}
\begin{aligned}
&\text{if no dividend is paid, the value of the }\\ 
&\qquad\qquad\qquad\qquad\qquad\text{resulting controlled process at time }T^\nu\text{ is }X^0_{T^\nu}, 
\end{aligned}
\label{setsumei}
\end{align}
where 
$X^0_t:= X_t - \rbra{\un{X}_t \land 0}$ 
for $t\geq 0 $.
Let us define a measurable function $g^z_f$ on $[0,z]\times E$ for $z \in[0, \infty)$ by 
$g^z_f(x, y)=x+ f(z-x, y)$. 
Then $x \mapsto g^z_f(x, y)$ has a left derivative, denoted by $x \mapsto g^{z\prime}_{f-}(x, y)$, which satisfies
$g^{z\prime}_{f-}(x, y) = 1-f^\prime_+(z-x, y)$ 
for $(x, y)\in[0,z]\times E$. 
Thus, we have, for $y\in E$, 
\begin{align}
g^{z\prime}_{f-}(x, y)
\begin{cases}
> 0, \qquad & x \in [0, z-\bar{\bfb}_f (y)), 
\\
= 0,\qquad & x \in [z-\bar{\bfb}_f (y), z-\un{\bfb}_f (y)], 
\\
< 0, \qquad &  x\in (z-\un{\bfb}_f (y), z]. 
\end{cases}
\label{Lem305_001}
\end{align}
Therefore, for $y \in E$, $x\mapsto g^{z}_{f}(x, y)$ is strictly increasing on $[0, z-\bar{\bfb}_f (y)]$ and strictly decreasing on $[z-\un{\bfb}_f (y), z]$, and thus for every ${y^\prime}\in E$ and for $m_{y^\prime}$-a.e. $y\in E$, 
\begin{align}
g^z_f((z-\bfb_f (y))\lor 0, y)\geq g^z_f(x, y),\qquad x \in [0,z].
\label{ineqg} 
\end{align}
By \eqref{setsumei}, the definition of $\pi^\nu_{\bfb_f}$ and \eqref{ineqg}, we have, for $\pi\in\un{\Pi}^\nu$, 
\begin{align}
L^{\pi^\nu_{\bfb_f}}_{T^\nu}+ f(U^{\pi^\nu_{\bfb_f}}_{T^\nu}, Y_{T^\nu})
=&
g^{X^0_{T^\nu}}_f((X^0_{T^\nu}-\bfb_f (Y_{T^\nu}))\lor 0, Y_{T^\nu}) \\
\geq &
g^{X^0_{T^\nu}}_f(L^\pi_{T^\nu}, Y_{T^\nu})
=L^{\pi}_{T^\nu}+ f(U^{\pi}_{T^\nu}, Y_{T^\nu}), \label{303_01}
\end{align}
for all $(x,y)\in\bR\times E$, $\bfP_{(x, y)}$-a.s. 
By \eqref{303_02} and \eqref{303_01}, 
we have for $\pi\in\un{\Pi}^\nu$ and $(x, y)\in\bR\times E$, 
\begin{align}
\cV^{\nu}_{\pi^\nu_{\bfb_f}}f(x, y)\geq \cV^{\nu}_{\pi}f(x, y).
\end{align}
The proof is complete. 
\end{proof}
To implement the DP, we use the fact that iterates of the operator $\cV^{\nu}$ applied to a particular function converge to the value function.  
To establish this, we need several lemmas, which are presented below.
\begin{Lem}\label{Lem306}
For $f \in \Gamma^\nu$, we have $\cV^{\nu} f \in\Gamma^\nu$. 
\end{Lem}
\begin{proof}
We first define, for $g\in \Gamma^\nu$ and $y \in E$,  
\begin{align}
\tilde{g} (x, y):= &
\begin{cases}
g(x, y), \qquad &x\leq \un{\bfb}_g(y), \\
g(\un{\bfb}_g(y), y)+ (x-\un{\bfb}_g(y)) , \qquad &x> \un{\bfb}_g(y), 
\end{cases}
\\
=&\begin{cases}
g(x, y), \qquad &x\leq \bar{\bfb}_g(y), \\
g(\bar{\bfb}_g(y), y)+ (x-\bar{\bfb}_g(y)) , \qquad &x> \bar{\bfb}_g(y). 
\end{cases}
\label{tilde_f}
\end{align}
We fix $f \in \Gamma^\nu$. 
Note that for $y \in E$, the function $x\mapsto\tilde{f}(x, y)$ is concave and the right derivative $x\mapsto\tilde{f}^\prime_+(x, y)$ 
is equal to $1$ for $x\geq \un{\bfb}_f(y)$ by the definition of $\un{\bfb}_f$. 
By Lemma \ref{Lem303}, we have, for $x \in\bR$, 
\begin{align}
\cV^{\nu}&f(x,y )= \cV^{\nu}_{{\pi^\nu_{\bar{\bfb}_f}}}f (x, y)
\\
&=
\bfE_{(x, y)}\sbra{ e^{-\frq(T^\nu)} L^{{\pi^\nu_{\bar{\bfb}_f}}}_{T^\nu}  -\beta\int_{[0,T^\nu]}e^{-\frq(t)} \diff R^{{\pi^\nu_{\bar{\bfb}_f}}}_t+ e^{-\frq(T^\nu)} f(U^{{\pi^\nu_{\bar{\bfb}_f}}}_{T^\nu}, Y_{T^\nu})}\\
&=
\bfE_{(x, y)}\sbra{   -\beta\int_{[0,T^\nu]}e^{-\frq(t)} \diff ((-\un{X}_t) \lor 0)+ e^{-\frq(T^\nu)} \tilde{f}(X^0_{T^\nu}, Y_{T^\nu})}.
\label{optimal_b}
\end{align}
\par
We confirm that $\cV^{\nu} f$ satisfies condition (iii) for membership in $\Gamma^\nu$. 
We define, for $(x,y)\in\bR\times E$,  
\begin{align}
v^L_0 (x, y):=& \bfE_{(x, y)}\sbra{\int_{[0,\infty) }e^{-\frq(t)} \diff L^{0}_t}
=x\lor0+v^L_{\pi^\nu_0}(0, y),  \\
v^R_0 (x, y):= &\bfE_{(x, y)}\sbra{\int_{[0,\infty) }e^{-\frq(t)} \diff R^{0}_t}, 
\end{align}
where $\pi^\nu_0$ is the MMPCB strategy with barrier $0$. 
The derivative of the map $x\mapsto v^L_0 (x, y)$ on $(0,\infty)$ is equal to $1$.  Hence, by the definition of $\tilde{f}$ and condition (iii) for membership in $\Gamma^\nu$,  
we have 
\begin{align}
v^L_0 (x, y) \geq \tilde{f}(x,y)\geq -\beta v^R_0 (x, y),\qquad (x, y) \in \bR \times E. 
\end{align} 
Therefore, by \eqref{optimal_b}, \eqref{periodicMarkov} at $T^\nu$ and \eqref{upperbound_by_0strategy}, we have, for $(x, y)\in [0,\infty)\times E$, 
 \begin{align}
\cV^{\nu} f (x, y) &\leq \bfE_{(x, y)}\sbra{e^{-\frq(T^\nu)} \tilde{f}(X^0_{T^\nu}, Y_{T^\nu})} \\
&\leq\bfE_{(x, y)}\sbra{e^{-\frq(T^\nu)} v^L_0 (X^0_{T^\nu}, Y_{T^\nu})} 
=v^L_{\pi^\nu_0}(x, y)\leq v^L_0(x, y).
\end{align}
On the other hand, by \eqref{optimal_b}, \eqref{periodicMarkov} at $T^\nu$ and \eqref{upperbound_by_0strategy}, 
\begin{align}
\cV^{\nu} f (0, y)& =
\bfE_{(0, y)}\sbra{   -\beta\int_{[0,T^\nu]}e^{-\frq(t)} \diff ((-\un{X}_t) \lor 0)+ e^{-\frq(T^\nu)} \tilde{f}(X^0_{T^\nu}, Y_{T^\nu})} \\
&\geq 
\bfE_{(0, y)}\sbra{   -\beta\int_{[0,T^\nu]}e^{-\frq(t)} \diff R^{\pi^\nu_0}_t-\beta e^{-\frq(T^\nu)} v^R_0(0, Y_{T^\nu})} 
=-\beta v^R_{\pi^\nu_0}(0,y)
\geq-\beta v^R_0(0, y). 
\end{align}
\par
We confirm condition (i) for membership in $\Gamma^\nu$. 
We compute the right derivative $x\mapsto {(\cV^{\nu}f)}^\prime_+(x, y)$ 
for fixed $y\in E$. 
For $\varepsilon\in\bR$, we define $X^{(\varepsilon)}:=\{X^{(\varepsilon)}_t : t\geq 0\}$ as 
\begin{align}
X^{(\varepsilon)}_t := X_t+\varepsilon, \qquad t\geq 0, 
\end{align}
and define, for measurable function $\bfb$,
\begin{align}
&\ \ \ \tau^{(\varepsilon), -}_0:=\inf\{t>0:X^{(\varepsilon)}_t<0\},\qquad  
\tau^{(\varepsilon), -}_{0+}:=\inf\{t\geq 0:X^{(\varepsilon)}_t\leq0\},
\\
T^{(\varepsilon),+}_\bfb:&=\min\{T^\nu_k>0: X^{(\varepsilon)}_{T^\nu_k}>\bfb{(Y_{T^\nu_k})} \}, 
\qquad 
T^{(\varepsilon),+}_{\bfb-}:=\min\{T^\nu_k>0: X^{(\varepsilon)}_{T^\nu_k}\geq \bfb{(Y_{T^\nu_k})} \}, \\
&\qquad \qquad 
\un{X}^{(\varepsilon)}_t := \inf_{s \in [0, t]}X^{(\varepsilon)}_s,
\qquad 
X^{0,(\varepsilon)}_t:=X^{(\varepsilon)}_t-(\un{X}^{(\varepsilon)}_t \land 0)
\qquad t\geq 0. 
\end{align} 
We also write $L^{\pi^{(\varepsilon)}_{\bfb}}:=\{L^{\pi^{(\varepsilon)}_{\bfb}}_t : t\geq 0\}$, $R^{\pi^{(\varepsilon)}_{\bfb}}:=\{R^{\pi^{(\varepsilon)}_{\bfb}}_t : t\geq 0\}$ and $U^{\pi^{(\varepsilon)}_{\bfb}}:=\{U^{\pi^{(\varepsilon)}_{\bfb}}_t : t\geq 0\}$ for processes representing the dividend, capital injection and controlled processes when the MMPCB strategy with barrier $\bfb$ is applied to $X^{(\varepsilon)}$. 
Then, by \eqref{200_01}, the following assertion ($\sharp$) can be proved by adapting the proof of 
\cite[Appendix C]{Nob2021}, which treats the case of L\'evy processes. 
\begin{itemize}
\item[($\sharp$)]
We fix a measurable function $\bfb$. 
For $\varepsilon>0$ and $(x, y)\in\bR\times E$, we have, $\bfP_{(x, y)}$-a.s.,  
\begin{align}
L^{{\pi}^{(\varepsilon)}_{\bfb}}_t- L^{{\pi}^{(0)}_{\bfb}}_t=0, \ 
R^{{\pi}^{(\varepsilon)}_{\bfb}}_t- R^{{\pi}^{(0)}_{\bfb}}_t=0, \ 
 U^{{\pi}^{(\varepsilon)}_{\bfb}}_t-U^{{\pi}^{(0)}_{\bfb}}_t= \varepsilon, 
\quad t \in [0, {T}^{ (\varepsilon), +}_\bfb \land {\tau}^{ (0), -}_0), 
\end{align}
and the process $t\mapsto  L^{{\pi}^{(\varepsilon)}_{\bfb}}_t-L^{{\pi}^{(0)}_{\bfb}}_t\in[0, \varepsilon]$ is non-decreasing and the processes 
$t\mapsto R^{{\pi}^{(\varepsilon)}_{\bfb}}_t-R^{{\pi}^{(0)}_{\bfb}}_t\in[-\varepsilon, 0]$ and $t\mapsto U^{{\pi}^{(\varepsilon)}_{\bfb}}_t-U^{{\pi}^{(0)}_{\bfb}}_t\in[0,\varepsilon]$ are non-increasing on $(0, \infty)$. 
In particular, if ${T}^{ (0), +}_{\bfb-} \leq {\tau}^{ (0), -}_0$,
\begin{align}
L^{{\pi}^{(\varepsilon)}_{\bfb}}_t- L^{{\pi}^{(0)}_{\bfb}}_t=\varepsilon,
 \quad 
R^{{\pi}^{(\varepsilon)}_{\bfb}}_t- R^{{\pi}^{(0)}_{\bfb}}_t=0,\quad 
 U^{{\pi}^{(\varepsilon)}_{\bfb}}_t-U^{{\pi}^{(0)}_{\bfb}}_t= 0, 
\quad t \in [{T}^{ (0), +}_{\bfb-}, \infty) , 
\end{align}
and if $ {\tau}^{ (\varepsilon), -}_{0+}<{T}^{ (\varepsilon), +}_\bfb$,
\begin{align}
L^{{\pi}^{(\varepsilon)}_{\bfb}}_t- L^{{\pi}^{(0)}_{\bfb}}_t=0,
 \quad 
R^{{\pi}^{(\varepsilon)}_{\bfb}}_t- R^{{\pi}^{(0)}_{\bfb}}_t=-\varepsilon,\quad 
 U^{{\pi}^{(\varepsilon)}_{\bfb}}_t-U^{{\pi}^{(0)}_{\bfb}}_t= 0, 
\quad t \in [ {\tau}^{ (\varepsilon), -}_{0+}, \infty), 
\end{align}
\end{itemize}
Then, for $x \in \bR$ and $\varepsilon > 0$,  we have
\begin{align}
&\cV^{\nu}f(x+\varepsilon ,y )- \cV^{\nu}f(x,y )\\
&\ \ =\bfE_{(x, y)}\Bigg{[} e^{-\frq(T^\nu)} 
\rbra{ \tilde{f}(X^{0,(\varepsilon)}_{T^\nu}, Y_{T^\nu})-\tilde{f}(X^{0,(0)}_{T^\nu}, Y_{T^\nu})}
+{\beta}\int_{[0,T^\nu]}e^{-\frq(t)} \diff 
(((-X_t)\lor 0)\land \varepsilon)
\Bigg{]},\label{Lem306_001}
\end{align}
where we used \eqref{optimal_b}. 
By ($\sharp$), we have
\begin{align}
\bfE_{(x, y)}&\sbra{ e^{-\frq(T^\nu)} 
\rbra{ \tilde{f}(X^{(\varepsilon)}_{T^\nu}, Y_{T^\nu})-\tilde{f}(X^{(0)}_{T^\nu}, Y_{T^\nu})}; \un{X}^{(0)}_{T^\nu}\geq  0
}\\
&\leq
\bfE_{(x, y)}\sbra{ e^{-\frq(T^\nu)} 
\rbra{ \tilde{f}(X^{0,(\varepsilon)}_{T^\nu}, Y_{T^\nu})-\tilde{f}(X^{0,(0)}_{T^\nu}, Y_{T^\nu})}}\\
&\leq\bfE_{(x, y)}\sbra{ e^{-\frq(T^\nu)} 
\rbra{ \tilde{f}(X^{(\varepsilon)}_{T^\nu}, Y_{T^\nu})-\tilde{f}(X^{(0)}_{T^\nu}
, Y_{T^\nu})}; \un{X}^{(\varepsilon)}_{T^\nu}> 0
},
\label{Lem306_002}
\end{align} 
and 
\begin{align}
\beta\varepsilon
\bfE_{(x, y)}\sbra{e^{-\frq(\tau^{(\varepsilon),-}_{0+})};\un{X}^{(\varepsilon)}_{T^\nu}\leq 0}
&\leq 
\bfE_{(x, y)}\sbra{ {\beta}\int_{[0,T^\nu]}e^{-\frq(t)} \diff 
(((-X_t)\lor 0)\land \varepsilon)}
\\
&\leq 
\beta\varepsilon
\bfE_{(x, y)}\sbra{e^{-\frq(\tau^{(0),-}_0)};\un{X}^{(0)}_{T^\nu}< 0
}.
\label{Lem306_003}
\end{align}
Combining \eqref{Lem306_001}, \eqref{Lem306_002} and \eqref{Lem306_003}, and using the fact that $\lim_{\varepsilon\downarrow0}\tau^{(\varepsilon),-}_{0+}=\tau^{(0),-}_0$, $\bfP_{(x, y)}$-a.s., we have,
for $x\in\bR$, 
\begin{align}
{(\cV^{\nu}f)}^\prime_+(x ,y )&:=
\lim_{\varepsilon\downarrow0}
\frac{\cV^{\nu}f(x+\varepsilon ,y )- \cV^{\nu}f(x,y )}{\varepsilon}\\
&=\bfE_{(x, y)}\sbra{ e^{-\frq(T^\nu)} 
 \tilde{f}^\prime_+(X_{T^\nu}, Y_{T^\nu}); \un{X}^{(0)}_{T^\nu}\geq  0
}
 +\beta
\bfE_{(x, y)}\sbra{e^{-\frq(\tau^{-}_0)};\un{X}^{(0)}_{T^\nu}<  0 
}\\
 &=\bfE_{(0, y)}\sbra{ e^{-\frq(T^\nu\land\tau^{(x),-}_0)} 
\rbra{ \tilde{f}^\prime_+(X^{(x)}_{T^\nu}, Y_{T^\nu})1_{\{\un{X}^{(x)}_{T^\nu}\geq  0
\}} 
+\beta1_{\{\un{X}^{(x)}_{T^\nu}<  0
\}}}}.
 \label{Lem306_004}
\end{align} 
Since $x\mapsto e^{-\frq(T^\nu\land\tau^{(x),-}_0)} $ and $x\mapsto  \tilde{f}^\prime_+(X^{(x)}_{T^\nu}, Y_{T^\nu})1_{\{\un{X}^{(x)}_{T^\nu}\geq  0
\}} 
+\beta1_{\{\un{X}^{(x)}_{T^\nu}<  0
\}}$ are non-negative and non-increasing, the right derivative $x\mapsto {(\cV^{\nu}f)}^\prime_+(x ,y )$ 
exists and is also non-increasing. 
It follows from this fact and \cite[Theorem 6.4]{HirLem2001} that the function $x\mapsto \cV^{\nu}  f(x, y)$ is concave. 
\par
By \eqref{Lem306_004} and since 
$\tilde{f}^\prime_+(x, y)\in[1, \beta]$ for $(x, y)\in[0, \infty)\times E$, 
$\cV^\nu f$ satisfies condition (ii) for membership in $\Gamma^\nu$. 
The proof is complete. 
\end{proof}
From Lemma \ref{Lem306}, we can iterate $\cV^{\nu}$ on functions in $\Gamma^\nu$. Thus, we can define, for $n\in\bN$ and $f \in\Gamma^\nu$,
\begin{align}
\cV^{\nu,(n)}f (x, y):=\cV^{\nu}(\cV^{\nu,(n-1)}f) (x, y),\qquad (x, y)\in\bR\times E,
\end{align}
inductively, where $\cV^{\nu, (0)}f(x, y):=f(x, y)$ for $(x, y)\in\bR\times E$. Note that $\cV^{\nu, (1)}f(x, y)=\cV^{\nu}f(x, y)$ for $(x, y)\in\bR\times E$. 
\par
We define the measurable function 
\begin{align}
v_0(x, y):=-\beta\bfE_{(x, y)}\sbra{\int_{[0, \infty)} e^{-\frq(t)} \diff ((-\un{X}_t) \lor 0) },\qquad (x, y)\in \bR\times E.
\end{align}
Then, the function $v_0$ belongs to $\Gamma^\nu$. 
Indeed, $v_0$ satisfies condition (iii) for membership in $\Gamma^\nu$ by \eqref{upperbound_by_0strategy}. In addition, by considering the case $T^{\nu} = \infty$ in the proof of Lemma \ref{Lem306}, we immediately see that the right derivative $v^\prime_{0+}$ of $x \mapsto v_0(x, y)$ coincides with $x\mapsto \beta\bfE_{(x,y)}\sbra{e^{-\frq(\tau^-_0)}}$ for $y\in E$, and $v_0$ satisfies the remaining conditions for membership in $\Gamma^\nu$. 
For a strategy $\pi\in\un{\Pi}^\nu$ and $n\in\bN$, we define the strategy $\pi_n\in \un{\Pi}^\nu$ as 
\begin{align}
L^{\pi_n}_t=
L^{\pi}_{t\land T^\nu_n}
,\qquad t\geq 0. 
\label{def_pi_n}
\end{align}
\begin{Lem}\label{Lem307}
For any strategy $\pi\in\un{\Pi}^\nu$ and $n\in\bN$, we have
\begin{align}
v_{\pi_n}(x, y)\leq \cV^{\nu, (n)} {v_0}(x, y)\leq V^{\nu}(x, y),\qquad (x, y)\in\bR\times E. \label{inq_of_NPV}
\end{align} 
\end{Lem}
\begin{proof}
We prove the second inequality of \eqref{inq_of_NPV}. 
By the definition of $\cV^{\nu, (n)} $, Lemma \ref{Lem303} and \eqref{periodicMarkov} at each $T^\nu_k$ with $k\in\{1, 2, \dots, n\}$, we have   
\begin{align}
\cV^{\nu, (n)} {v_0}(x, y) = v_{\pi^\ast_n}(x, y),\qquad (x, y)\in\bR\times E, \label{Lem307_003}
\end{align}
where 
$\pi^\ast_n\in\un{\Pi}^\nu$ 
is a strategy defined analogously to ($\star$) in Section \ref{SubSec203}, where at each time $T^\nu_k$, a dividend is paid corresponding to the excess over $\un{\bfb}_{\cV^{\nu,(n-k)}v_0}(Y_{T^\nu_k})$, for $k = 1, \dots, n$. 
Thus, by the definition of $V^\nu$, the second inequality of \eqref{inq_of_NPV} holds.
\par
We prove the first inequality of \eqref{inq_of_NPV}. 
We prove by induction that, for $k\in \{ 0, 1, \dots, n\}$, the following holds
\begin{align}
v_{\pi_n}(x, y)\leq \bfE_{(x, y)}\Bigg{[}\int_{[0,T^\nu_{n-k}]}e^{-\frq(t)} \diff (L^{\pi_n}_t
&-\beta R^{\pi_n}_t)
+e^{-\frq(T^\nu_{n-k})}\cV^{\nu, (k)}{v_0}(U^{\pi_n}_{T^\nu_{n-k}}, Y_{T^\nu_{n-k}})\Bigg{]}, 
\label{induction_inequality}
\end{align}
for $ (x, y)\in\bR\times E$. 
We have \eqref{induction_inequality} with equality in place of  $\leq$ when $k=0$ by \eqref{periodicMarkov} at $T^\nu_n$. 
We assume that $n\geq 2$ and \eqref{induction_inequality} holds for $k=l \in\{0, 1, \dots, n-1\}$ and prove it for $k=l+1$. 
By \eqref{induction_inequality} with $k=l$, 
we have 
\begin{align}
v_{\pi_n}(x, y)\leq\bfE_{(x, y)}\Bigg{[}\int_{[0,T^\nu_{n-l-1}]}e^{-\frq(t)} \diff L^{\pi_n}_t&+ e^{-\frq(T^\nu_{n-l})} \Delta L^{\pi_n}_{T^\nu_{n-l}}
\\
-\beta\int_{[0,T^\nu_{n-l}]}e^{-\frq(t)} \diff R^{\pi_n}_t
&+e^{-\frq(T^\nu_{n-l})}\cV^{\nu, (l)}{v_0}(U^{\pi_n}_{T^\nu_{n-l}}, Y_{T^\nu_{n-l}})\Bigg{]}
,
\label{Lem307_001}
\end{align}
for $(x,y)\in\bR\times E$.  
As in the proof of \eqref{303_01}, we have 
\begin{align}
\cV^{\nu, (l)}{v_0}(U^{\pi_n}_{(n-l-)}\land \un{\bfb}_{\cV^{\nu, (l)}v_0}(Y_{T^\nu_{n-l}}), Y_{T^\nu_{n-l}}) 
+&\left(U^{\pi_n}_{(n-l-)}-\un{\bfb}_{\cV^{\nu, (l)}v_0}(Y_{T^\nu_{n-l}})\right)\lor 0\\
&\geq
\Delta L^{\pi_n}_{T^\nu_{n-l}}+\cV^{\nu, (l)} {v_0}(U^{\pi_n}_{T^\nu_{n-l}}, Y_{T^\nu_{n-l}}). 
\label{Lem307_002}
\end{align} 
Thus, we have 
\begin{align}
&v_{\pi_n}(x, y)\leq\bfE_{(x, y)}\Bigg{[}\int_{[0,T^\nu_{n-l-1}]}e^{-\frq(t)} \diff L^{\pi_n}_t 
-\beta\int_{[0,T^\nu_{n-l}]}e^{-\frq(t)} \diff R^{\pi_n}_t+e^{-\frq(T^\nu_{n-l})}\\
&\quad\times\left(\left(U^{\pi_n}_{(n-l-)}-\un{\bfb}_{\cV^{\nu, (l)}v_0}(Y_{T^\nu_{n-l}})\right)\lor 0\right)
+e^{-\frq(T^\nu_{n-l})} \cV^{\nu, (l)}{v_0}(U^{\pi_n}_{(n-l-)}\land \un{\bfb}_{\cV^{\nu, (l)}v_0}(Y_{T^\nu_{n-l}}), Y_{T^\nu_{n-l}})\Bigg{]}\\
&=\bfE_{(x, y)}\sbra{\int_{[0,T^\nu_{n-l-1}]}e^{-\frq(t)} \diff (L^{\pi_n}_t
-\beta R^{\pi_n}_t) 
+e^{-\frq(T^\nu_{n-l-1})}\cV^{\nu, (l+1)}{v_0}(U^{\pi_n}_{T^\nu_{n-l-1}}, Y_{T^\nu_{n-l-1}})},
\end{align}
where the inequality follows from \eqref{Lem307_001} and \eqref{Lem307_002}, and the equality follows from \eqref{periodicMarkov} at $T^\nu_{n-l-1}$.
Therefore, we obtain \eqref{induction_inequality} with $k=l+1$.
Since \eqref{induction_inequality} with $k=n$ is equivalent to the first inequality of \eqref{inq_of_NPV}, 
the proof is complete. 
\end{proof}
\begin{Lem}\label{Lem308}
For a strategy $\pi\in\un{\Pi}^\nu$, we have
\begin{align}
\lim_{n\to\infty}\absol{v_\pi(x, y)-v_{\pi_n}(x, y)}=0,
\qquad (x, y)\in\bR\times E. 
\end{align}
\end{Lem}
Since the proof of this lemma is easy, we include it in Appendix \ref{lemma_proof} of the Supplementary Material. 
\begin{Lem}\label{Lem309}
We have
\begin{align}
\lim_{n\to\infty}\cV^{\nu, (n)} {v_0}(x, y)= V^{\nu}(x, y),\qquad (x, y)\in\bR\times E. 
\label{Lem309_statement}
\end{align}
In addition, the function $V^{\nu}$ belongs to $\Gamma^\nu$. 
\end{Lem}
\begin{proof}
By \eqref{inq_of_NPV}, we have 
\begin{align}
\limsup_{n\to\infty}\cV^{\nu, (n)} {v_0}(x, y)\leq  V^{\nu}(x, y),\qquad (x, y)\in\bR\times E. \label{Lem309_004}
\end{align}
For $(x, y)\in\bR\times E$ and $\varepsilon>0$, we may choose the strategy $\pi^{(x, y),\varepsilon}\in\un{\Pi}^\nu$ such that 
\begin{align}
0\leq V^{\nu}(x, y)-v_{\pi^{(x, y),\varepsilon}}(x, y)\leq\varepsilon. 
\label{Lem309_001}
\end{align}
For $n\in\bN$, $(x, y)\in\bR\times E$ and $\varepsilon>0$, we have 
\begin{align}
V^{\nu}(x, y)-\cV^{\nu, (n)} {v_0}(x, y)\leq &V^{\nu}(x, y)-v_{\pi^{(x, y),\varepsilon}_n}(x, y)\\
\leq&V^{\nu}(x, y)-v_{\pi^{(x, y),\varepsilon}}(x, y)
+|v_{\pi^{(x, y),\varepsilon}}(x, y)-v_{\pi^{(x, y),\varepsilon}_n}(x, y)|
\\
\leq&\varepsilon
+|v_{\pi^{(x, y),\varepsilon}}(x, y)-v_{\pi^{(x, y),\varepsilon}_n}(x, y)|, 
\label{Lem309_002}
\end{align}
where the first inequality follows from Lemma \ref{Lem307} and the last inequality follows from \eqref{Lem309_001}. 
By taking the limit of \eqref{Lem309_002} as $n\to \infty$ and by Lemma \ref{Lem308}, we have
\begin{align}
V^{\nu}(x, y)-\liminf_{n\to\infty}\cV^{\nu, (n)} {v_0}(x, y)
\leq\varepsilon. \label{Lem309_003}
\end{align}
Since \eqref{Lem309_003} holds for any $\varepsilon>0$ and $(x,y)\in\bR\times E$, we have 
\begin{align}
V^{\nu}(x, y)-\liminf_{n\to\infty}\cV^{\nu, (n)} {v_0}(x, y)
\leq0, \qquad (x, y)\in\bR\times E. \label{Lem309_005}
\end{align}
By \eqref{Lem309_004} and \eqref{Lem309_005}, we obtain \eqref{Lem309_statement}. 
Since $\cV^{\nu, (n)} {v_0} \in \Gamma^\nu$ for $n\in\bN$, it follows from 
\eqref{Lem309_statement} and \cite[Theorem B.3.1.4]{HirLem2001} that $V^\nu$ satisfies conditions (i) and (iii) for membership in $\Gamma^\nu$. 
Using also \cite[Theorem 1.1]{Lac1982}, we see that $V^\nu$ satisfies condition (ii) for membership in $\Gamma^\nu$. 
The proof is complete. 
\end{proof}
The next lemma, which can be derived using the lemmas established so far, completes \textbf{Step 2}. 
From Lemma \ref{Lem309}, we can define the non-negative measurable functions $\un{\bfb}_{V^{\nu}} $ and $\bar{\bfb}_{V^{\nu}} $ on $E$. 
Let $\Pi^{\nu,\ast}_{V^\nu}$ be the set of strategies $\pi$ obtained from the definition of $\Pi^{\nu,\ast}$ by replacing $\un{\bfb}^\nu$ with $\un{\bfb}_{V^{\nu}} $ and $\bar{\bfb}^\nu$ with $\bar{\bfb}_{V^{\nu}} $. 
Let $\Xi^\nu_{V^\nu}$ and $\hat{\Xi}^\nu_{V^\nu}$ be the sets of non-negative measurable functions $\bfb$ on $E$ satisfying 
$E^\bfb_{V^\nu}= E$ and $m_y (E\backslash E^\bfb_{V^\nu})=0$ for all $y\in E$, respectively, where 
$E^\bfb_{V^\nu}:=\cbra{y \in E: \un{\bfb}_{V^{\nu}} (y)\leq\bfb(y) \leq \bar{\bfb}_{V^{\nu}} (y)}$.
\begin{Lem}\label{Lem311}
The strategy $\pi\in\un{\Pi}^\nu$ is optimal if and only if it belongs to $\Pi^{\nu,\ast}_{V^\nu}$.
\end{Lem}
Since the implication from left to right is not difficult, though somewhat technical, it is given in Appendix \ref{lemma_proof} of the Supplementary Material. We prove here only the converse implication.
\begin{proof}[Proof of ($\Leftarrow$) in Lemma \ref{Lem311}]
(1) For $\bfb_{V^\nu}\in\hat{\Xi}^\nu_{V^\nu}$, we prove
\begin{align}
\cV^{\nu}_{\pi^\nu_{\bfb_{V^{\nu}}}}V^{\nu}(x, y)= V^{\nu}(x, y),\qquad (x, y)\in\bR\times E. \label{Lem310_003}
\end{align}
By the definition of $\cV^\nu$ and \eqref{inq_of_NPV}, we have 
\begin{align}
\cV^\nu_{\pi^\nu_{\bfb_{V^\nu}}} \cV^{\nu, (n)} {v_0}(x, y) 
\leq  \cV^{\nu, (n+1)} {v_0}(x, y)
\leq V^{\nu}(x, y),\qquad (x, y)\in\bR\times E. 
\label{Lem310_005}
\end{align}
By the dominated convergence theorem, \eqref{Lem309_statement} and \eqref{Lem310_005}, we have 
\begin{align}
\cV^\nu_{\pi^\nu_{\bfb_{V^{\nu}}}} V^{\nu}(x, y)
=\lim_{n\to\infty} \cV^\nu_{\pi^\nu_{\bfb_{V^\nu}}} \cV^{\nu, (n)} {v_0}(x, y) 
\leq V^{\nu}(x, y),\qquad (x, y)\in\bR\times E. \label{Lem310_001}
\end{align}
Let $\tilde{\Pi}^\nu$ be the set of strategies $\pi$ belonging to $\un{\Pi}^\nu$ such that 
\begin{align}
L^{\pi}_t=\sum_{k\in\bN}
\left((  U^\pi_{(k-)} - \bfb^{\pi}_k(Y_{T^\nu_k}) )\lor 0\right) 1_{[0,t]}(T^\nu_k)
,\qquad t\geq 0, 
\end{align}
for a set of non-negative measurable functions $\{\bfb^{\pi}_n:n\in\bN\}$ on $E$. 
Since $\pi^\ast_n$ defined in the proof of Lemma \ref{Lem307} belongs to $\tilde{\Pi}^\nu$,  \eqref{Lem307_003} and \eqref{Lem309_statement} imply that 
\begin{align}
V^{\nu}(x, y ) = \sup_{\pi\in\tilde{\Pi}^\nu}v_\pi(x, y), \qquad (x, y)\in\bR\times E. 
\end{align}
For $n \in\bN$ and a strategy $\pi \in \tilde{\Pi}^\nu$, we define 
a strategy $\pi_{(n)} \in \tilde{\Pi}^\nu$ to satisfy $\bfb^{\pi_{(n)}}_k\equiv \bfb^{\pi}_{k+n}$ for $k\in\bN$. 
Then, we have, for $\pi \in \tilde{\Pi}^\nu$,
\begin{align}
\cV^{\nu}_{\pi^\nu_{\bfb_{V^{\nu}}}}V^{\nu}(x, y)=\cV^{\nu}V^{\nu}(x, y)\geq \cV^{\nu} v_{\pi_{(1)}}(x, y)
\geq v_\pi(x, y),\qquad (x, y)\in\bR\times E. 
\end{align}
By taking the supremum over $\pi$, we have 
\begin{align}
\cV^{\nu}_{\pi^\nu_{\bfb_{V^{\nu}}}}V^{\nu}(x, y)\geq V^{\nu}(x, y),\qquad (x, y)\in\bR\times E. \label{Lem310_002}
\end{align}
By \eqref{Lem310_001} and \eqref{Lem310_002}, we obtain \eqref{Lem310_003}. 
\par
(2) We prove the optimality of the strategies included in $\Pi^{\nu,\ast}_{V^\nu}$.
We fix $\pi\in\Pi^{\nu,\ast}_{V^\nu}$. 
Using the definitions of $\un{\bfb}_{V^\nu}$, $\bar{\bfb}_{V^\nu}$, and $\hat{\Xi}^\nu_{V^\nu}$ together with \eqref{Lem310_003} and \eqref{optimal_b}, we have, for $(x, y)\in \bR\times E$, 
\begin{align}
V^\nu(x, y) =& \bfE_{(x, y)}\sbra{   -\beta\int_{[0,T^\nu]}e^{-\frq(t)} \diff ((-\un{X}_t) \lor 0)+ e^{-\frq(T^\nu)} \tilde{V}^\nu(X^0_{T^\nu}, Y_{T^\nu})}\\
=& \bfE_{(x, y)}\sbra{e^{-\frq(T^\nu)} \Delta L^\pi_{T^\nu}   -\beta\int_{[0,T^\nu]}e^{-\frq(t)} \diff R^{\pi}_t+ e^{-\frq(T^\nu)} V^\nu(U^\pi_{T^\nu}, Y_{T^\nu})}. 
\end{align}
Iterating this identity $n\in\bN$ times and using \eqref{periodicMarkov}, we have
\begin{align}
V^\nu(x, y) =& \bfE_{(x, y)}\sbra{\int_{[0, T^\nu_n]}e^{-\frq(t)}  \diff L^\pi_t   -\beta\int_{[0,T^\nu_n]}e^{-\frq(t)} \diff R^{\pi}_t+ e^{-\frq(T^\nu_n)} V^\nu(U^\pi_{T^\nu_n}, Y_{T^\nu_n})}. 
\end{align}
Thus, we have, for $(x, y)\in\bR\times E$ and $n\in\bN$, 
\begin{align}
&0\leq V^{\nu}(x, y)-v_{\pi}(x, y)\\
&=\bfE_{(x, y)}\sbra{ e^{-\frq(T^\nu_n)} v^L_0(U^\pi_{T^\nu_n}, Y_{T^\nu_n})
-\int_{[T^\nu_{n+1}, \infty)}e^{-\frq(t)}  \diff L^\pi_t   +\beta\int_{(T^\nu_n, \infty) }e^{-\frq(t)} \diff R^{\pi}_t
}\\
&\leq\bfE_{(x, y)}\sbra{ e^{-\frq(T^\nu_n)} 
\rbra{
V^\nu(U^\pi_{T^\nu_n}, Y_{T^\nu_n})+\beta v^R_0(U^\pi_{T^\nu_n}, Y_{T^\nu_n})
}}
\leq\bfE_{(x, y)}\sbra{ e^{-\frq(T^\nu_n)} \rbra{U^{\pi}_{T^\nu_n}+(1+\beta)B^\nu}}, 
\label{divide}
\end{align}
where the second inequality follows from \eqref{periodicMarkov} and \eqref{upperbound_by_0strategy}, 
and the last inequality follows from \eqref{finiteness}. 
For sufficiently large $\ell\in\bN$ such that $
e^{-q(k-1)}k\leq e^{-\frac{q\ell}{2}}$ for $k>\ell$,  
we have
\begin{align}
\bfE_{(x, y)}\sbra{ e^{-\frq(T^\nu_n)} U^{\pi}_{T^\nu_n}}&\leq 
\bfE_{(0, y)}\sbra{e^{-q T^\nu_n}\rbra{(x\lor0)+X_{T^\nu_n}-\un{X}_{T^\nu_n}}} \\
&\leq 
2\bfE_{(0, y)}\sbra{e^{-q T^\nu_n}\rbra{(x\lor0)+\sup_{s\in[0,T^\nu_n]}|X_{s}|}} \\
&\leq 
2\sum_{k\in\bN} e^{-q(k-1)} \bfE_{(0, y)}\sbra{(x\lor0)+\sup_{s\in[0,k]}|X_{s}|; T^\nu_n\in (k-1, k]} \\
&\leq 
2\sum_{k\in\bN} e^{-q(k-1)} 
((x\lor0)+kM)
\bfP(T^\nu_n\in (k-1, k])\\
&\leq 
2\sum_{k=1}^\ell e^{-q(k-1)} 
((x\lor0)+kM)
\bfP(T^\nu_n\in (k-1, k])\\
&\ \ + 
2e^{-q\ell} (x\lor0)\bfP(T^\nu_n\in (\ell, \infty))+2e^{-\frac{q\ell}{2}} 
M\bfP(T^\nu_n\in (\ell, \infty)), 
\end{align}
where the first inequality follows from the fact that $\pi\in\un{\Pi}^\nu$. 
By taking the limit as $n\to\infty$, we have 
\begin{align}
\limsup_{n\to\infty}\bfE_{(x, y)}\sbra{ e^{-\frq(T^\nu_n)} U^{\pi}_{T^\nu_n}}\leq 
2e^{-q\ell} (x\lor0)+2e^{-\frac{q\ell}{2}} M. 
\label{limsup_bound}
\end{align}
Since \eqref{limsup_bound} is true for any large enough $\ell\in\bN$, we have 
\begin{align}
\limsup_{n\to\infty}\bfE_{(x, y)}\sbra{ e^{-\frq(T^\nu_n)} U^{\pi}_{T^\nu_n}}=0. \label{limsup_result}
\end{align}
By \eqref{divide} with $n\to \infty$ and \eqref{limsup_result}, we have, for $(x, y)\in\bR\times E$,
$V^{\nu}(x, y)-v_{\pi}(x, y)=0$. 
The proof is complete. 
\end{proof}
Combining Lemma \ref{Lem304a} and Lemma \ref{Lem311}, we have obtained a version of Theorem \ref{newThm302} in which $\Pi^{\nu,\ast}$ is replaced by $\Pi^{\nu,\ast}_{V^\nu}$. Then, we have also obtained a version of Corollary \ref{Thm302} in which $\hat{\Xi}^\nu$ is replaced by $\hat{\Xi}^\nu_{V^\nu}$.  
\par
We next prove that $\bar{\bfb}_{V^\nu}$ is finite and then proceed to Step 3. Before doing so, we first present a lemma on the right derivatives of the expected NPVs of dividends and capital injections under the MMPCB strategy, which will be needed in the proof of finiteness and in Step 3. 
\begin{Lem}\label{right_derivative_lemma}
For any non-negative measurable function $\bfb$ on $E$ and $(x, y)\in\bR\times E$, we have 
\begin{align}
v^\prime_{{\pi}^{\nu}_{\bfb}+}(x, y)=\bfE_{(x,y )}\sbra{ e^{-\frq({T}^{ \nu, +}_{\bfb-})} ;{T}^{ \nu, +}_{\bfb-} \leq {\tau}^{  -}_0}+
\beta\bfE_{(x,y )}\sbra{ e^{-\frq({\tau}^{  -}_0)} ;{\tau}^{  -}_0<{T}^{ \nu, +}_{\bfb-} }.
\label{right_derivative}
\end{align}
\end{Lem}
Since the proof of this lemma is almost the same as that of \cite[Proposition 4.1]{MatNobPer2025+}, we include it in Appendix \ref{lemma_proof} of the Supplementary Material. 
\begin{Lem}\label{Lem312}
The function $\bar{\bfb}_{V^\nu}$ is finite. 
\end{Lem}
Since the proof of this lemma is easy, we include it in Appendix \ref{lemma_proof} of the Supplementary Material. 
\par
If we can then show that $\Pi^{\nu,\ast}$ and $\Pi^{\nu,\ast}_{V^\nu}$ are equal, we can substitute $\Pi^{\nu,\ast}_{V^\nu}$ for $\Pi^{\nu,\ast}$, $\Xi^\nu_{V^\nu}$ for $\Xi^\nu$, $\hat{\Xi}^\nu_{V^\nu}$ for $\hat{\Xi}^\nu$, $\un{\bfb}_{V^\nu}$ for $\un{\bfb}^\nu$, and $\bar{\bfb}_{V^\nu}$ for $\bar{\bfb}^\nu$ in the results obtained so far, and thus obtain Theorems \ref{Thm302a} and \ref{newThm302}, and Corollary \ref{Thm302}.
\par
\textbf{Step 3.}
We show that $\Xi^\nu$ and $\Xi^\nu_{V^\nu}$ are the same, and consequently that $\hat{\Xi}^\nu$ and $\hat{\Xi}^\nu_{V^\nu}$ are the same. 
As a consequence, we establish the two identities
$\un{\bfb}^\nu=\un{\bfb}_{V^\nu}$ and $\bar{\bfb}^\nu=\bar{\bfb}_{V^\nu}$,
as well as Theorems \ref{Thm302a} and \ref{newThm302} and Corollary \ref{Thm302}.
For this purpose, we use the following three lemmas.  
\begin{Lem}\label{LemA01}
We assume that the non-negative measurable function $\bfb$ on $E$ satisfies $v_{\pi^\nu_\bfb}\in\Gamma^\nu$ and $\bfb(y)\in[\un{\bfb}_{v_{\pi^\nu_\bfb}}(y),\bar{\bfb}_{v_{\pi^\nu_\bfb}}(y)]$ for $y\in E$. Then, $v_{\pi^\nu_\bfb}$ is equal to $V^\nu$. 
\end{Lem}

\begin{proof}
Since $\bfb(y)\in[\un{\bfb}_{v_{\pi^\nu_\bfb}}(y),\bar{\bfb}_{v_{\pi^\nu_\bfb}}(y)]$ for $y\in E$ and by \eqref{periodicMarkov} at $T^\nu$, we have, for $(x, y) \in \bR\times E$, 
\begin{align}
\cV^{\nu}v_{\pi^\nu_\bfb}(x, y)=
\cV^{\nu}_{\pi^\nu_{\bfb}}v_{\pi^\nu_\bfb}(x, y)
=v_{\pi^\nu_\bfb}(x, y). \label{tokucho}
\end{align}
For $\pi\in\un{\Pi}^\nu$, $n\in\bN$ and $f\in\Gamma^\nu$, we write
\begin{align}
\cV^{\nu,(n)}_\pi f (x, y):=\cV^{\nu}_\pi(\cV^{\nu,(n-1)}_\pi f) (x, y),\qquad (x, y)\in\bR\times E,
\end{align}
inductively, if it is well-defined, where $\cV^{\nu, (0)}_\pi f(x, y):=f(x, y)$ for $(x, y)\in\bR\times E$. 
It is easy to verify that for $n\in\bN$, $\cV^{\nu, (n)}_{\pi^\nu_{\un{\bfb}_{V^\nu}}}v_{\pi^\nu_\bfb}$ is well-defined and 
\begin{align}
\cV^{\nu, (n)}_{\pi^\nu_{\un{\bfb}_{V^\nu}}}v_{\pi^\nu_\bfb}(x, y)
\leq v_{\pi^\nu_\bfb}(x, y),\qquad (x, y)\in\bR\times E, \label{pib_ineq}
\end{align} 
by induction. 
In fact, by the definition of $\cV^\nu$ and \eqref{tokucho}, we have
\begin{align}
\cV^{\nu}_{\pi^\nu_{\un{\bfb}_{V^\nu}}}v_{\pi^\nu_\bfb}(x, y)
\leq \cV^\nu v_{\pi^\nu_\bfb}(x, y)=v_{\pi^\nu_\bfb}(x, y), 
\end{align}
so $\cV^{\nu}_{\pi^\nu_{\un{\bfb}_{V^\nu}}}v_{\pi^\nu_\bfb}$ is well-defined and \eqref{pib_ineq} with $n=1$ holds. 
If $\cV^{\nu, (k)}_{\pi^\nu_{\un{\bfb}_{V^\nu}}}$ is well-defined and \eqref{pib_ineq} with $n=k$ is true, then by the assumption, the definition of $\cV^\nu$ and \eqref{tokucho},
\begin{align}
\cV^{\nu, (k+1)}_{\pi^\nu_{\un{\bfb}_{V^\nu}}}v_{\pi^\nu_\bfb}(x, y)
\leq \cV^{\nu}_{\pi^\nu_{\un{\bfb}_{V^\nu}}}v_{\pi^\nu_\bfb}(x, y)
\leq \cV^\nu v_{\pi^\nu_\bfb}(x, y)=v_{\pi^\nu_\bfb}(x, y), 
\end{align}
so $\cV^{\nu, (k+1)}_{\pi^\nu_{\un{\bfb}_{V^\nu}}}v_{\pi^\nu_\bfb}$ is well-defined and \eqref{pib_ineq} with $n=k+1$ holds. 
Since $v_{\pi^\nu_\bfb}$ is not greater than $V^\nu$, we have, for $n\in\bN$ and $(x, y)\in\bR\times E$, 
\begin{align}
&0\leq V^{\nu}(x, y) -v_{\pi^\nu_\bfb}(x, y)
\leq \cV^{\nu, (n)}_{\pi^\nu_{\un{\bfb}_{V^\nu}}}V^\nu(x, y) -\cV^{\nu, (n)}_{\pi^\nu_{\un{\bfb}_{V^\nu}}}v_{\pi^\nu_\bfb}(x, y) \\
&=\bfE_{(x, y)}\sbra{  e^{-\frq(T^\nu_n)} \rbra{V^\nu(U^{\pi^\nu_{\un{\bfb}_{V^{\nu}}}}_{T^\nu_n}, Y_{T^\nu_n})-v_{\pi^\nu_\bfb}(U^{\pi^\nu_{\un{\bfb}_{V^{\nu}}}}_{T^\nu_n}, Y_{T^\nu_n})}}
\leq\bfE_{(x, y)}\sbra{ e^{-\frq(T^\nu_n)} \rbra{U^{\pi^\nu_{\un{\bfb}_{V^{\nu}}}}_{T^\nu_n}+2B^\nu}}
. \label{ThmB01_002}
\end{align}
Here, the second inequality follows from \eqref{Lem310_003} and \eqref{pib_ineq}, the equality from \eqref{periodicMarkov}, and the last inequality from \eqref{finiteness}. 
By \eqref{ThmB01_002} with $n\to\infty$ and \eqref{limsup_result}, 
we have, for $(x, y)\in \bR \times E$, 
$V^{\nu}(x, y) -v_{\pi^\nu_\bfb}(x, y)=0$. 
The proof is complete. 
\end{proof}
\par
The following discussion introduces additional tools and two technical lemmas. The technicalities in this
part mainly arise from the possible lack of smoothness of the
additive component and can largely be avoided, for instance, when $X$ has a non-trivial Brownian component. Thus, on a first reading, it may be enough to focus on the main statements.
\par
Let $p$ be a measurable function from $E$ to $[0, 1]$. 
Let $\{\lambda^p(k): k\in\bN\cup\{0\}\}$ be the family of independent and identically distributed random variables, each uniformly distributed on $[0, 1]$, and assume that this family is independent of $\{X_t:t\geq 0\}$, $\{Y_t:t\geq 0\}$ and $N^\nu$.
\begin{align}
\Lambda^p(y, k) = 
\begin{cases}
0,\qquad &\text{if } \lambda^p(k) \leq p(y), \\
1,\qquad &\text{if }\lambda^p(k)> p(y).  
\end{cases}
\end{align}
We define, for $n\in\bN$ and non-negative measurable function $\bfb$, 
\begin{align}
\cT^{n}_{p, \bfb}:=\min\{T^\nu_k > \cT^{n-1}_{p, \bfb}: 
U^{\pi^{\nu}_\bfb}_{T^\nu_k} =\bfb(Y_{T^\nu_k}), \
\Delta L^{\pi^{\nu}_\bfb}_{T^\nu_k}+\Lambda^p(Y_{\cT^{n-1}_{p, \bfb}}, n-1)>0
\},
\end{align}
where $\cT^{0}_{p, \bfb}:=0$. 
We also define 
$\cK_\bfb(p, t):=\max\{k \in \bN:\cT^{k}_{p,\bfb}\leq t\}$ 
for $t\geq 0$,
where $\max\varnothing=0$, 
\begin{align}
\varkappa^{\nu,\bfb}_{p,0}:=\inf\{ t \geq 0: 
U^{\pi^{\nu}_\bfb}_{t} =0, \
\Lambda^p(Y_{\cT^{\cK_\bfb(p, t)}_{p,\bfb}}, \cK_\bfb(p, t))=0
\}\land\inf\{t>0: R^{\pi^{\nu}_\bfb}_{t}>0\},
\end{align}
and 
\begin{align}
\varrho^{\nu, \bfb}_{p}( x, y):= \beta \bfE_{(x, y)}\sbra{e^{-\frq(\varkappa^{\nu,\bfb}_{p,0})}}, \qquad (x, y)\in\bR\times E. 
\end{align}
Furthermore, we define
\begin{align}
\kappa^{\nu,\bfb}_{p, 0}:=\rbra{T^{\nu,+}_{\bfb-}+\varkappa^{\nu,\bfb}_{p,0}\circ\theta_{T^{\nu,+}_{\bfb-}}}\land\inf\{t>0: R^{\pi^{\nu}_\bfb}_{t}>0\}, 
\end{align}
and
\begin{align}
\rho^{\nu, \bfb}_{p}( x, y):= \beta \bfE_{(x, y)}\sbra{e^{-\frq(\kappa^{\nu,\bfb}_{p,0})}}, \qquad (x, y)\in\bR\times E. 
\end{align}
Here, since $\Lambda^p$ is independent of the process to which the shift is applied, we regard it as invariant under the shift operator $\theta_{T^{\nu,+}_{\bfb-}}$, that is,
$\Lambda^p( y, k)\circ \theta_{T^{\nu,+}_{\bfb-}}=\Lambda^p( y, k)$ for $y\in E$ and $k\in\bN\cup\{0\}$. 
For instance, we may take $\Lambda^p$ to be defined on an auxiliary probability space and work on the product space, with the shift operator acting only on the original coordinate and leaving $\Lambda^p$ unchanged. 
\begin{Lem}\label{LemB01a}
The map $x\mapsto\rho_p^{\nu, \bfb}( x, y)$ is non-increasing for $y\in E$. 
\end{Lem}
\begin{Lem}\label{LemB01}
For $\bfb \in \Xi^\nu$, there exists a measurable function $p_\bfb$ from $E$ to $[0, 1]$ such that for $y\in E$, 
\begin{align}
\varrho^{\nu, \bfb}_{p_\bfb}( \bfb (y), y) =1,\qquad 
\rho^{\nu, \bfb}_{p_\bfb}(  x, y)
\begin{cases}
\geq 1,\qquad &x <\bfb(y),\\
\leq 1 ,\qquad &x>\bfb(y). 
\end{cases}
\label{LemB01_main}
\end{align}
\end{Lem}
The proofs of these lemmas are rather technical. However, as mentioned above, they become much simpler when the additive component has a certain degree of smoothness. Since some readers may wish to skip these technical details, we defer the proofs to Appendix \ref{lemma_proof} of the Supplementary Material.
\par 
The following lemma is the main result of Step 3.
\begin{Lem}
$\Xi^\nu$ and $\Xi^\nu_{V^\nu}$ are equal. 
\end{Lem}
\begin{proof}
We prove (1) ``if $\bfb\in\Xi^\nu_{V^\nu}$, then $\bfb\in\Xi^\nu$'' and (2) ``if $\bfb\in\Xi^\nu$, then $\bfb\in\Xi^\nu_{V^\nu}$'' in turn. 
\par 
(1) We assume $\bfb\in\Xi^\nu_{V^\nu}$ and prove $\bfb\in\Xi^\nu$. 
By the definition of $\Xi^\nu_{V^\nu}$, we have, for $ y\in E$,   
\begin{align}
\lim_{\varepsilon\downarrow 0}\frac{v_{{\pi}^{\nu}_{\bfb}}(\bfb(y)+\varepsilon, y)-v_{{\pi}^{\nu}_{\bfb}}(\bfb(y), y)}{\varepsilon}
\leq 1\leq
\lim_{\varepsilon\downarrow 0}\frac{v_{{\pi}^{\nu}_{\bfb}}(\bfb(y), y)-v_{{\pi}^{\nu}_{\bfb}}(\bfb(y)-\varepsilon, y)}{\varepsilon}. 
 \label{maxMMPCB}
\end{align}
By \eqref{right_derivative}, we have, for $y\in E$, 
\begin{align}
\lim_{\varepsilon\downarrow0}&
\frac{v_{{\pi}^{\nu}_{\bfb}}(\bfb(y)+\varepsilon, y)-v_{{\pi}^{\nu}_{\bfb}}(\bfb(y), y)}{\varepsilon}\\
&=\bfE_{(\bfb(y),y )}\sbra{ e^{-\frq({T}^{ \nu, +}_{\bfb-})} ;{T}^{ \nu, +}_{\bfb-} \leq {\tau}^{  -}_0}+
\beta\bfE_{(\bfb(y),y )}\sbra{ e^{-\frq({\tau}^{  -}_0)} ;{\tau}^{  -}_0<{T}^{ \nu, +}_{\bfb-} }.
\label{right_derivative_02}
\end{align}
By employing an argument similar to the proof of Lemma \ref{right_derivative_lemma}, we have, for $ y \in E$, 
\begin{align}
\lim_{\varepsilon\downarrow0}&
\frac{v_{{\pi}^{\nu}_{\bfb}}(\bfb(y), y)-v_{{\pi}^{\nu}_{\bfb}}(\bfb(y)-\varepsilon, y)}{\varepsilon} 
=\lim_{\varepsilon\downarrow0}
\frac{v_{{\pi}^{(0)}_{\bfb}}(\bfb(y), y)-v_{{\pi}^{(-\varepsilon)}_{\bfb}}(\bfb(y), y) }{\varepsilon}
\\ 
 &=\bfE_{(\bfb(y),y )}\sbra{ e^{-\frq({T}^{ \nu, +}_{\bfb})} ;{T}^{ \nu, +}_{\bfb} < {\tau}^{  -}_{0+}}+
\beta\bfE_{(\bfb(y),y )}\sbra{ e^{-\frq({\tau}^{  -}_{0+})} ;{\tau}^{  -}_{0+}< {T}^{ \nu, +}_{\bfb} }.
\label{left_derivative}
\end{align}
By \eqref{maxMMPCB}, \eqref{right_derivative_02} and \eqref{left_derivative},
we obtain $\bfb\in\Xi^\nu$. 
\par
(2) We assume $\bfb\in\Xi^\nu$ and prove $\bfb\in\Xi^\nu_{V^\nu}$. 
By the proof of Lemma \ref{right_derivative_lemma}, we have, for $(x, y) \in \bR\times E$, 
\begin{align}
v^{L\prime}_{\pi^\nu_\bfb +}(x, y) := &\lim_{\varepsilon\downarrow0}\frac{v^L_{\pi^\nu_\bfb}(x+\varepsilon, y)-v^L_{\pi^\nu_\bfb}(x, y)}{\varepsilon}
=\bfE_{(x, y)}\sbra{e^{-\frq (T^{\nu,+}_{\bfb-})};T^{\nu,+}_{\bfb-} \leq\tau^-_0}, \\
v^{R\prime}_{\pi^\nu_\bfb +}(x, y) := &\lim_{\varepsilon\downarrow0}\frac{v^R_{\pi^\nu_\bfb}(x+\varepsilon, y)-v^R_{\pi^\nu_\bfb}(x, y)}{\varepsilon}
=-\bfE_{(x, y)}\sbra{e^{-\frq (\tau^-_0)};\tau^-_0<T^{\nu,+}_{\bfb-} }.
\end{align}
Both $x\mapsto v^{L\prime}_{\pi^\nu_\bfb +}(x, y)$ and $x\mapsto v^{R\prime}_{\pi^\nu_\bfb +}(x, y)$ are non-decreasing for any $y\in E$. 
Hence, by \cite[Theorem 6.4]{HirLem2001}, the functions $x\mapsto v^L_{\pi^\nu_\bfb}(x, y) $ and $x\mapsto v^R_{\pi^\nu_\bfb}(x, y) $ are convex, and their right derivatives are Radon--Nikodym densities of these functions, respectively, for any $y \in E$. 
Thus, for fixed $y\in E$, the function $x\mapsto v^\prime_{\pi^\nu_\bfb +}(x, y)$, defined by $v^\prime_{\pi^\nu_\bfb +}(x, y):=v^{L\prime}_{\pi^\nu_\bfb +}(x, y)-\beta v^{R\prime}_{\pi^\nu_\bfb +}(x, y)$, is a density of $x\mapsto v_{\pi^\nu_\bfb}(x , y)$. 
By the identity in \eqref{LemB01_main} and \eqref{periodicMarkov}, we have, for $(x, y)\in\bR\times E$,  
\begin{align}
&v^\prime_{\pi^\nu_\bfb +}(x, y)=\bfE_{(x, y)}\sbra{e^{-\frq (T^{\nu,+}_{\bfb-})}\varrho^{\nu, \bfb}_{p_\bfb}( \bfb (Y_{T^{\nu,+}_{\bfb-}}), Y_{T^{\nu,+}_{\bfb-}});T^{\nu,+}_{\bfb-} \leq\tau^-_0}
+\beta \bfE_{(x, y)}\sbra{e^{-\frq (\tau^-_0)};\tau^-_0<T^{\nu,+}_{\bfb-} }\\
&\ =\sum_{k\in\bN}\bfE_{(x, y)}\sbra{e^{-\frq (T^{\nu}_k)}\varrho^{\nu, \bfb}_{p_\bfb}( \bfb (Y_{T^{\nu}_k}), Y_{T^{\nu}_k});T^{\nu,+}_{\bfb-}=T^\nu_k, T^{\nu}_k \leq\tau^-_0}
+\beta \bfE_{(x, y)}\sbra{e^{-\frq (\tau^-_0)};\tau^-_0<T^{\nu,+}_{\bfb-} }\\
&\ =\sum_{k\in\bN}\beta\bfE_{(x, y)}\sbra{e^{-\frq (\kappa^{\nu,\bfb}_{p_\bfb,0})};T^{\nu,+}_{\bfb-}=T^\nu_k, T^{\nu}_k \leq\tau^-_0}
+\beta \bfE_{(x, y)}\sbra{e^{-\frq (\tau^-_0)};\tau^-_0<T^{\nu,+}_{\bfb-} }
=\rho^{\nu, \bfb}_{p_\bfb}(  x, y).
\label{v_right_derivative}
\end{align}  
Since $x\mapsto v^\prime_{\pi^\nu_\bfb+}(x, y)$ is non-increasing by \eqref{v_right_derivative} and Lemma \ref{LemB01a}, 
$x\mapsto v_{\pi^\nu_\bfb}(x, y)$ is concave. 
In addition, $v_{\pi^\nu_\bfb}$ satisfies condition (ii) for membership in $\Gamma^\nu$ by \eqref{v_right_derivative} and condition (iii) for membership in $\Gamma^\nu$ by \eqref{upperbound_by_0strategy}. Hence, $v_{\pi^\nu_\bfb}$ belongs to $\Gamma^\nu$. 
By Lemma \ref{LemB01}, we have $\bfb(y)\in[\un{\bfb}_{v_{\pi^\nu_\bfb}}(y), \bar{\bfb}_{v_{\pi^\nu_\bfb}}(y)]$ for $y\in E$. 
By Lemma \ref{LemA01}, $\pi^\nu_\bfb$ is an optimal strategy
and thus we have $\un{\bfb}_{v_{\pi^\nu_\bfb}}(y)=\un{\bfb}_{V^\nu}(y)$ and $ \bar{\bfb}_{v_{\pi^\nu_\bfb}}(y)=\bar{\bfb}_{V^\nu}(y)$ for $y \in E$, which implies $\bfb\in\Xi^\nu_{V^\nu}$. The proof is complete. 
\end{proof}
Finally, assuming that $\nu$ has an exponential distribution with intensity $r>0$, we prove Theorem \ref{Thm303}. 
To that end, 
we prove the following lemma.
\begin{Lem}\label{Lem317}
We fix a non-negative measurable function $\bfb$ on $E$ and 
a measurable function $p$ from $E$ to $[0,1]$. 
The map $t\mapsto \rho^{\nu, \bfb}_p (X_t, Y_t)$ is right-continuous on $[0,  \tau^-_{0+})$, $\bfP_{(x, y)}$-a.s. for any $(x, y)\in(0, \infty)\times E$. 
\end{Lem}
We give the proof in Appendix \ref{lemma_proof} of the Supplementary Material, since it is essentially a calculation based on the observation that, from a suitable point of view, $\rho^{\nu, \bfb}_p $ is an excessive function.
\begin{proof}[Proof of Theorem \ref{Thm303}]
It suffices to prove that, for each fixed $a> 0$, 
\begin{align}
B_a:=\{y\in E: \un{\bfb}^\nu(y)>a\} \text{ is finely open with respect to }Y. \label{finely_open}
\end{align}
Fix $y \in B_a$. Then, we can also choose $x \in (a, \un{\bfb}^\nu(y))$ and $\tilde{\Omega}\in\cF$ such that $\bfP_{(x, y)}(\tilde{\Omega})=1$ and 
$t \mapsto (X_t(\omega), Y_t(\omega))$ and $t\mapsto \rho^{\nu,\un{\bfb}^\nu}_{p_{\un{\bfb}^\nu}} (X_t(\omega), Y_t(\omega))$ are both right-continuous on $[0, \tau^-_{0+})$ for all $\omega\in\tilde{\Omega}$ by Lemma \ref{Lem317}. 
By the definitions of $x$, $y$ and $\tilde{\Omega}$, together with the identity $\un{\bfb}^\nu=\un{\bfb}_{V^\nu}$ and the fact that $x\mapsto \rho^{\nu, \un{\bfb}^\nu}_{p_{\un{\bfb}^\nu}}( x, y)$ is the density of $x\mapsto V^\nu( x, y)$, 
we have $\rho^{\nu, \un{\bfb}^\nu}_{p_{\un{\bfb}^\nu}}(x, y)>1$, 
and for $\omega\in \tilde{\Omega}$, we may choose $\delta>0$ such that 
\begin{align}
X_t(\omega)>a, \quad \rho^{\nu,\un{\bfb}^\nu}_{p_{\un{\bfb}^{\nu}}} (X_t(\omega), Y_t(\omega))>1 \qquad t \in [0, \delta), 
\end{align}
which implies
\begin{align}
\un{\bfb}^\nu( Y_t(\omega))>a, \text{ and thus }Y_t(\omega)\in  B_a,
\text{ for }t\in[0, \delta).  
\end{align}
Therefore, we have 
\begin{align}
\bfP^{Y,(x)}_{y}(T_{ E\backslash B_a}>0) \geq \bfP^{Y,(x)}_{y}(\tilde{\Omega})=1,
\end{align}
where $T_A:=\inf\{t>0:Y_t \in A \}$ for $A\in\cB(E)$. This proves \eqref{finely_open}. 
Since \eqref{finely_open} holds for any $a \in \bR$, the function $\un{\bfb}^\nu$ is finely lower semi-continuous. 
\par
The proof of the fine upper semi-continuity of $\bar{\bfb}^\nu$ is analogous, with the relevant inequalities reversed, and is therefore omitted.
\par
The proof is complete. 
\end{proof}

\section{Approximation of the classical--classical setting}
\label{Sec_approximation}
In the previous section, we proceeded with the discussion in the setting of Section \ref{PCcase}.
In this section, we study the value function $V$ in the classical--classical setting of Section \ref{CCcase}, by approximating it by the value functions in the periodic--classical setting of Section \ref{PCcase}. This suggests the form of an optimal strategy in the classical--classical setting of Section \ref{CCcase}. We consider two kinds of approximations.

\subsection{Approximation by constant-interval periodic dividends}
For each $n\in\bN$, we write $\Pi^{\frac{1}{2^n}}$ for the set of admissible strategies in the periodic--classical setting, where $\nu=\delta_{\frac{1}{2^n}}$. 
Then, it is obvious that $\Pi^{\frac{1}{2^m}}\subset \Pi^{\frac{1}{2^n}}\subset \Pi$ for $m,n\in\bN$ with $m<n$. 
Thus, we also have, for $m,n\in\bN$ with $m<n$ and $(x, y)\in\bR\times E$,
\begin{align}
V^{\frac{1}{2^m}}(x, y)\leq V^{\frac{1}{2^n}}(x, y)\leq
V(x, y) ,  \label{value_function_ineq}
\end{align}
where $V^{\frac{1}{2^n}}(x, y)=\sup_{\pi\in\Pi^{\frac{1}{2^n}}}v_\pi(x, y)$ for $n\in\bN$ and $(x, y)\in\bR\times E$. 
\par
For $\pi \in \Pi$ and $(x, y)\in\bR\times E$, we may choose the sequence of strategies $\{\pi_n : n\in\bN\}$ such that $\pi_n\in\Pi^{\frac{1}{2^n}}$ for $n\in\bN$ and 
$\lim_{n\to\infty} v_{\pi_n}(x, y) =v_\pi (x, y)$. 
In fact, we define $\pi_n$ for $n\in\bN$ to satisfy
\begin{align}
L^{\pi_n}_t := L^\pi_{[t]_{\frac{1}{2^n}}-} ,\qquad R^{\pi_n}_t:=R^\pi_t,\qquad t\geq 0. 
\end{align}
Here $[t]_{\frac{1}{2^n}}:=\max\cbra{ \frac{k}{2^n}: k\in \bN, \frac{k}{2^n}\leq t}$ 
with $\max\varnothing=-\infty$ and $ L^\pi_{-\infty}=0$. 
Then the sequence $\{\pi_n : n\in\bN\}$ satisfies the above conditions. 
This fact and \eqref{value_function_ineq} imply the following proposition. 
\begin{Prop} \label{Prop401}
For $(x, y)\in\bR\times E$, we have 
$V^{\frac{1}{2^n}}(x, y) \uparrow V(x, y)$ 
as $n\uparrow\infty$. 
Thus, $x \mapsto V(x, y)$ is also concave for $y\in E$. 
\end{Prop}
From the proofs of Theorems \ref{Thm302a} and \ref{Thm302}, 
there exists a lower endpoint of the interval of the optimal barriers, denoted by $\un{\bfb}_{V^{\frac{1}{2^n}}}$.
For these functions, we can obtain the following theorem. 
\begin{Thm}\label{barrier_approximation_1}
For $ y\in E$, the sequence $\cbra{\un{\bfb}_{V^{\frac{1}{2^n}}}( y):n\in\bN}$ is non-decreasing. 
\end{Thm}
The above theorem and Proposition \ref{Prop401} suggest that, in the setting of Section \ref{CCcase}, an optimal strategy may be to pay dividends whenever the capital level exceeds the barrier determined by the limit of $\un{\bfb}_{V^{\frac{1}{2^n}}}$, while making capital injections only when the capital level falls below $0$. 
By an argument similar to that in the proof of Theorem \ref{barrier_approximation_1}, we can also conclude that the sequence $\{\bar{\bfb}_{V^{\frac{1}{2^n}}}(y): n\in\bN\}$ is non-decreasing. 

\begin{proof}[Proof of Theorem \ref{barrier_approximation_1}]
By the definition of $\un{\bfb}_{V^{\frac{1}{2^n}}}$, it suffices to prove that, for $n\in\bN$,  
\begin{align}
V^{\frac{1}{2^n}\prime}_{+}(x, y)
\leq V^{\frac{1}{2^{n+1}}\prime}_{+}(x, y),\qquad (x,y)\in \bR\times E, \label{barrier_result}
\end{align}
where $V^{\frac{1}{2^n}\prime}_+$ is the right derivative of $V^{\frac{1}{2^n}}$. 
\par
When $\nu=\delta_{\frac{1}{2^n}}$, 
we use the notation obtained by replacing $\nu$ with $\frac{1}{2^n}$ in the symbols introduced in Section \ref{Optimality_pcb}.
In addition, we write $V^{\frac{1}{2^n}}_{ (k)} $ for $\cV^{\frac{1}{2^n}, (k)} {v_0}$ for $n\in\bN$ and $k\in\bN\cup\{0\}$. 
Since $x\mapsto V^{\frac{1}{2^n}}(x, y)$ and $x\mapsto V^{\frac{1}{2^n}}_{(k)}(x, y)$ are concave for $n, k\in\bN$ and $y \in E$, Lemma \ref{Lem309}, \cite[Theorem B.4.2.3]{HirLem2001} and \cite[Theorem 1.1]{Lac1982} 
imply $\lim_{k\to\infty}{V^{\frac{1}{2^n}\prime}_{(k)+}}(x, y)={V^{\frac{1}{2^n}\prime}_+}(x, y)$ for Lebesgue-a.e. $x$ for any $n\in\bN$ and $y \in E$. 
Therefore, if we show that for $n \in \bN$, 
\begin{align}
V^{\frac{1}{2^n}\prime}_{(k)+}(x, y)
\leq V^{\frac{1}{2^{n+1}}\prime}_{(2k)+}(x, y),\qquad k\in\bN\cup\{0 \},~(x,y)\in \bR\times E, \label{barrier_convergence}
\end{align}
then \eqref{barrier_result} follows.
In the following, we prove \eqref{barrier_convergence} by induction on $k$. 
\par 
For $k=0$, since both ${V^{\frac{1}{2^n}}_{(0)}}$ and ${V^{\frac{1}{2^{n+1}}}_{(0)}}$ equal $v_0$, \eqref{barrier_convergence} holds. 
Assuming that \eqref{barrier_convergence} holds for $k=l\in\bN\cup\{0\}$, we show that \eqref{barrier_convergence} also holds for $k=l+1$. 
For $V^{\frac{1}{2^{n}}}_{(l)}$ and $V^{\frac{1}{2^{n+1}}}_{(2l)}$, we define $\tilde{V}^{\frac{1}{2^{n}}}_{(l)}$ and $\tilde{V}^{\frac{1}{2^{n+1}}}_{(2l)}$ in the same way as \eqref{tilde_f}. 
Then, it follows that
\begin{align}
\tilde{V}^{\frac{1}{2^n}\prime}_{(l)+}(x, y)
\leq \tilde{V}^{\frac{1}{2^{n+1}}\prime}_{(2l)+}(x, y),\qquad (x,y)\in \bR\times E, 
\label{ass_inequality}
\end{align}
\par
For $f\in\Gamma^{\frac{1}{2^{n+1}}}$ and $(x,y)\in\bR\times E$, 
we define an operator $\cW^{\frac{1}{2^{n+1}}}$ on $\Gamma^{\frac{1}{2^{n+1}}}$ by 
\begin{align}
\cW^{\frac{1}{2^{n+1}}} f (x, y)= 
\bfE_{(x, y)}\sbra{  -\beta\int_{[0,\frac{1}{2^{n+1}}]}e^{-\frq(t)} \diff \rbra{(-\un{X}_t)\lor0}+ e^{-\frq(\frac{1}{2^{n+1}})} f(X^0_{\frac{1}{2^{n+1}}}, Y_{\frac{1}{2^{n+1}}})}. 
\end{align}
Note that the operator $\cW^{\frac{1}{2^{n+1}}}$ coincides with the operator $\cV^{\frac{1}{2^{n+1}}}_\pi$ when $\pi$ pays no dividends. 
Focusing on condition (iii) in the definition of the class $\Gamma^\nu$, a slight extension of \eqref{upperbound_by_0strategy} shows that $\Gamma^{\frac{1}{2^n}}\subset \Gamma^{\frac{1}{2^{n+1}}}$ holds, and that $\cW^{\frac{1}{2^{n+1}}}$ can also be applied to the functions in $\Gamma^{\frac{1}{2^n}}$. 
By applying almost the same method as in the proof of \eqref{Lem306_004}, we obtain the following expression for the right derivative of $x \mapsto \cW^{\frac{1}{2^{n+1}}}f(x, y)$, denoted by ${(\cW^{\frac{1}{2^{n+1}}}f)}^\prime_+$, for $(x,y)\in\bR\times E$,  
\begin{align}
{(\cW^{\frac{1}{2^{n+1}}}f)}^\prime_+&(x ,y )
=
\bfE_{(x, y)}\sbra{ e^{-\frq(\frac{1}{2^{n+1}}\land\tau^{-}_0)} 
\rbra{ f^\prime_+(X_{\frac{1}{2^{n+1}}}, Y_{\frac{1}{2^{n+1}}})1_{\{\un{X}_{\frac{1}{2^{n+1}}}\geq 0\}} 
+\beta1_{\{\un{X}_{\frac{1}{2^{n+1}}}<0\}}}}. \label{1/2_operator_derivative}
\end{align}
By applying the same method as in the proof of Lemma \ref{Lem306}, we can verify that $\cW^{\frac{1}{2^{n+1}}}f\in \Gamma^{\frac{1}{2^{n+1}}}$ for $f \in \Gamma^{\frac{1}{2^{n+1}}}$. 
By \eqref{Lem306_004} and \eqref{1/2_operator_derivative}, we have, for $f, g\in \Gamma^{\frac{1}{2^{n+1}}}$ with $f^\prime_+(x, y)\leq g^\prime_+(x, y)$ for $(x,y)\in\bR\times E$,
\begin{align}
{(\cW^{\frac{1}{2^{n+1}}}f)}^\prime_+(x ,y )\leq {(\cV^{\frac{1}{2^{n+1}}}g)}^\prime_+(x ,y ),
\qquad (x,y)\in\bR\times E. 
\label{compare_constant}
\end{align} 
\par
By \eqref{ass_inequality}, \eqref{1/2_operator_derivative} and the definition of $V^{\frac{1}{2^{n+1}}}_{(2l+1)}$, we have 
\begin{align}
{(\cW^{\frac{1}{2^{n+1}}}\tilde{V}^{\frac{1}{2^n}}_{(l)})}^\prime_+(x ,y )\leq
{(\cW^{\frac{1}{2^{n+1}}}\tilde{V}^{\frac{1}{2^{n+1}}}_{(2l)})}^\prime_+(x ,y )
=V^{\frac{1}{2^{n+1}}\prime}_{(2l+1)+}(x, y),
\quad (x, y)\in\bR\times E. \label{nagaifutoushiki}
\end{align}
By the Markov property at $\frac{1}{2^{n+1}}$, we have, for $(x,y)\in\bR\times E$, 
\begin{align}
V^{\frac{1}{2^n}}_{(l+1)}(x, y)=\cW^{\frac{1}{2^{n+1}}}\cW^{\frac{1}{2^{n+1}}}\tilde{V}^{\frac{1}{2^n}}_{(l)}(x, y). \label{3bunkai}
\end{align}
Using \eqref{3bunkai}, the fact that $V^{\frac{1}{2^{n+1}}}_{(2l+2)}$ is equal to $\cV^{\frac{1}{2^{n+1}}}V^{\frac{1}{2^{n+1}}}_{(2l+1)}$, \eqref{nagaifutoushiki} and \eqref{compare_constant}, we have  \eqref{barrier_convergence} with $k=l+1$. 
The proof is complete. 
\end{proof}
By \eqref{barrier_result} and since $\lim_{n\to\infty} V^{\frac{1}{2^n}\prime}_+(x, y)=V^\prime_+(x,y)$ for Lebesgue-a.e. $x$ for any $y \in E$ by Proposition \ref{Prop401}, \cite[Theorem B.4.2.3]{HirLem2001} and \cite[Theorem 1.1]{Lac1982}, we have, for $y \in E$,  
\begin{align}
\un{\bfb}_V(y):=\sup \{x \geq 0: V^\prime_+(x, y)>1\}=\lim_{n\to\infty}\un{\bfb}_{V^{\frac{1}{2^n}}}( y). \label{Sec401last}
\end{align}
From this and the discussion preceding the proof of Theorem \ref{barrier_approximation_1}, it is anticipated that the classical--classical setting should be analogous to the periodic--classical setting. 

\begin{Rem}
In the case of L\'evy processes, the approximation described in this section should allow us to recover the main result of \cite{Nob2021}. Indeed, one may consider the convergence, as $n\to\infty$, of the resulting controlled process under the periodic--classical barrier strategy with barrier $\un{\bfb}_{V^{\frac{1}{2^n}}}$ to the resulting controlled process under the double barrier strategy with barrier $\un{\bfb}_{V}$, and then combine this convergence with standard calculations. 
The same reasoning applies to the finite-state MAP setting. In particular, the result obtained in \cite{MatNobPerYam2024} should also be recoverable from this approximation. 
For a general MAP, however, the dependence of $\un{\bfb}_{V^{\frac{1}{2^n}}}(y)$ on $n\in\bN$ and $y\in E$ may be complicated, making it difficult to discuss convergence of the resulting controlled processes. 
Therefore, a precise description of the optimal strategy requires either establishing rigorous convergence estimates or solving an auxiliary impulse control problem first and then approximating the original problem. 
This is left for future work.
\end{Rem}

\subsection{Approximation by Poissonian-interval periodic dividends}
\label{Approximation_from_P}
For each $n\in\bN$, we write $\Pi^{\rmP, n}$ for the set of admissible strategies in the periodic--classical setting, where $\nu$ is the exponential distribution with intensity $n$. 
Let $\{N^{(k)}: k\in\bN\}$ be the family of independent Poisson processes $N^{(k)}=\{N^{(k)}_t:t\geq 0\}$ with intensity $1$, and assume that this family is independent of $\{X_t:t\geq 0\}$ and $\{Y_t:t\geq 0\}$.
Then, the times at which dividends under $\Pi^{\rmP, n}$ are paid can be regarded as the jump times of the Poisson process $\hat{N}^{(n)}=\{\hat{N}^{(n)}_t:t\geq 0\}$, where $\hat{N}^{(n)}_t=\sum_{k=1}^nN^{(k)}_t$ for $t\geq 0$. 
We adopt this representation throughout this section. 
We write $V^{\rmP, n}(x, y)=\sup_{\pi\in\Pi^{\rmP, n}} v_\pi(x, y)$ for $n\in\bN$ and $(x, y)\in\bR\times E$. 
Fix $m,n\in\bN$ with $m<n$. We may assume that $\cF$ contains the randomness generated by $\{N^{(k)}: k\in\bN\}$, and that both $\hat{N}^{(m)}_t$ and $\hat{N}^{(n)}_t$ are $\cF_t$-measurable for $t\geq 0$.
Since the jump times of $\hat{N}^{(m)}$ are included in those of $\hat{N}^{(n)}$ and 
by the definition of $V$, we have, for $(x, y)\in\bR\times E$,
\begin{align}
V^{\rmP, m}(x, y)\leq V^{\rmP,n}(x, y)\leq V(x, y)  . \label{P_value_function_ineq}
\end{align}
\par
For $\pi \in \Pi$ and $(x, y)\in\bR\times E$, we may choose the sequence of strategies $\{\pi_n : n\in\bN\}$ such that $\pi_n\in\Pi^{\rmP, n}$ for $n\in\bN$ and 
$\lim_{n\to\infty} v_{\pi_n}(x, y) =v_\pi (x, y)$. 
In fact, we define $\pi_n$ for $n\in\bN$ to satisfy
\begin{align}
L^{\pi_n}_t := L^\pi_{[t]_{\rmP , n}-} ,\qquad R^{\pi_n}_t:=R^\pi_t,\qquad t\geq 0. 
\end{align}
Here $[t]_{\rmP, n}=\max \{s\in[0, t]: \Delta \hat{N}^{(n)}_s \neq 0\}$ with $\max\varnothing=-\infty$
and $L^\pi_{-\infty}=0$. 
Then the sequence $\{\pi_n : n\in\bN\}$ satisfies the above conditions. 
This fact and \eqref{P_value_function_ineq} imply the following proposition. 
\begin{Prop}\label{Prop403}
For $(x, y)\in\bR\times E$, we have 
$V^{\rmP, n}(x, y) \uparrow V(x, y)$ 
as $n\uparrow\infty$. 
Thus, $x \mapsto V(x, y)$ is also concave for $y\in E$. 
\end{Prop}
Similarly to Theorem \ref{barrier_approximation_1}, we obtain the following theorem for the lower barriers $\{\un{\bfb}_{V^{\rmP, n}}:n\in\bN\}$ of the set of barriers corresponding to optimal MMPCB strategies.
\begin{Thm}\label{barrier_approximation_P}
For $ y\in E$, the sequence $\cbra{\un{\bfb}_{V^{\rmP, n}}( y):n\in\bN}$ is non-decreasing. 
\end{Thm}
Since the proof of this theorem is somewhat more involved than that of Theorem \ref{barrier_approximation_1} but follows similar lines, we give it in Appendix \ref{Sec00C} of the Supplementary Material. 
Combining this theorem with Proposition \ref{Prop403} suggests that, in the setting of Section \ref{CCcase}, the optimal strategy may be inferred from the limit of $\un{\bfb}_{V^{\rmP,n}}$, as in Theorem \ref{barrier_approximation_1}. 
A similar result holds for $\cbra{\bar{\bfb}_{V^{\rmP, n}}( y):n\in\bN}$.
As in \eqref{Sec401last}, we have, for $y \in E$,  
$\un{\bfb}_V(y):=\sup \{x \geq 0: V^\prime_+(x, y)>1\}=\lim_{n\to\infty}\un{\bfb}_{V^{\rmP, n}}( y)$. 

\section*{Acknowledgments}
The author was supported by JSPS KAKENHI grant no. JP21K13807 and JSPS Open Partnership Joint Research Projects grant no. JPJSBP120209921. 
In addition, from February 2023 to February 2025, the author stayed at Centro de Investigaci\'on en Matem\'aticas (CIMAT), Mexico, as a JSPS Overseas Research Fellow. The author is grateful for the stimulating research environment at CIMAT, which helped develop the main ideas of this paper.

\appendix

\section{The proofs of several lemmas in Section \ref{Optimality_pcb}} \label{lemma_proof}
\begin{proof}[Proof of Lemma \ref{Lem304}]
We write the resulting controlled process of the strategy $(L^0, R^0)$ as $U^0:=\{U^0_t : t\geq 0\}$. 
First, we prove in (1) that for $\pi\in\un{\Pi}^\nu$, 
\begin{align}
L^\pi_t \leq L^0_t,\qquad R^\pi_t \leq R^0_t,\qquad t\geq 0,
\label{0strategybound}
\end{align}
for all $(x,y)\in \bR \times E$, $\bfP_{(x, y)}$-a.s. 
Then, we obtain \eqref{upperbound_by_0strategy} since we have
\begin{align}
&\int_{[0,\infty)}e^{-\frq(t)} \diff L^0_t-\int_{[0,\infty)}e^{-\frq(t)} \diff L^\pi_t
=\int_0^\infty \bfq(Y_s) e^{-\frq(s)} ( L^0_s-L^\pi_s)\diff s \geq 0, \label{daishou} \\
&\int_{[0,\infty)}e^{-\frq(t)} \diff R^0_t-\int_{[0,\infty)}e^{-\frq(t)} \diff R^\pi_t
=\int_0^\infty \bfq(Y_s) e^{-\frq(s)} ( R^0_s-R^\pi_s)\diff s \geq 0, \label{daishou_002}
\end{align}
for all $(x,y)\in\bR\times E$, $\bfP_{(x, y)}$-a.s.
Second, we prove \eqref{finiteness} in (2). 
\par
(1) We prove \eqref{0strategybound} and 
\begin{align}
U^\pi_t \geq U^0_t, \qquad t\geq 0 \label{Ubound}
\end{align}
for all $(x,y)\in\bR\times E$, $\bfP_{(x, y)}$-a.s., 
for a fixed $\pi\in\un{\Pi}^\nu$, by induction. 
For $k \in \bN\cup\{0\}$, we assume that \eqref{0strategybound} and \eqref{Ubound} hold on $[0, T^\nu_k)$, and prove that \eqref{0strategybound} and \eqref{Ubound} hold on $[T^\nu_k, T^\nu_{k+1})$. 
By the assumption and \eqref{300_01}, we have 
\begin{align}
R^\pi_{T^\nu_k}=R^\pi_{T^\nu_k-}
-\rbra{U^\pi_{(k-)}\land 0}
\leq 
R^0_{T^\nu_k-}
-\cbra{(U^0_{T^\nu_k-}+\Delta X_{T^\nu_k} )\land 0}
=R^0_{T^\nu_k}, \label{Lem303_01}
\end{align}
for all $(x,y)\in\bR\times E$, $\bfP_{(x, y)}$-a.s., 
where $U^\pi_{(0-)}$ and $\Delta X_0$ are set equal to $X_0$, and $R^\pi_{0-}$, $R^0_{0-}$ and $U^0_{0-}$ are equal to $0$. 
By the definition of $U^0$, we have 
\begin{align}
U^\pi_{T^\nu_k}\geq 0 = U^0_{T^\nu_k},\label{Lem303_02}
\end{align}
for all $(x,y)\in\bR\times E$, $\bfP_{(x, y)}$-a.s. 
By \eqref{Lem303_01} and \eqref{Lem303_02}, we have
\begin{align}
L^\pi_{T^\nu_k}
=X_{T^\nu_k}+R^\pi_{T^\nu_k}-U^\pi_{T^\nu_k}
\leq X_{T^\nu_k}+R^0_{T^\nu_k}-U^0_{T^\nu_k}
=L^0_{T^\nu_k},
\end{align}
for all $(x,y)\in\bR\times E$, $\bfP_{(x, y)}$-a.s., 
and thus \eqref{0strategybound} and \eqref{Ubound} hold at $T^\nu_k$. 
By the definition of $L^0$, $R^0$ and $\pi$ and since \eqref{0strategybound} and \eqref{Ubound} hold at $T^\nu_k$, we have, for $t \in (T^\nu_k, T^\nu_{k+1})$,  
\begin{align}
L^\pi_t =L^\pi_{T^\nu_k}\leq L^0_{T^\nu_k} =L^0_t, 
\end{align}
\begin{align}
R^\pi_t =R^\pi_{T^\nu_k}-\cbra{\rbra{U^\pi_{T^\nu_k}+\inf_{s \in [T^\nu_k, t]}(X_s-X_{T^\nu_k})}\land0} \leq R^0_{T^\nu_k}-\cbra{\rbra{\inf_{s \in [T^\nu_k, t]}(X_s-X_{T^\nu_k})}\land0} =R^0_{t}, 
\end{align}
and 
\begin{align}
U^\pi_t=U^\pi_{T^\nu_k} +(X_t-X_{T^\nu_k})
-&\cbra{\rbra{U^\pi_{T^\nu_k}+\inf_{s \in [T^\nu_k, t]}(X_s-X_{T^\nu_k})}\land0}\\
&\geq (X_t-X_{T^\nu_k})-\cbra{\rbra{\inf_{s \in [T^\nu_k, t]}(X_s-X_{T^\nu_k})}\land0} =U^0_t,
\end{align}
for all $(x,y)\in\bR\times E$, $\bfP_{(x, y)}$-a.s. 
Thus, \eqref{0strategybound} and \eqref{Ubound} also hold for $t \in (T^\nu_k, T^\nu_{k+1})$. 
\par
(2) Based on the preceding argument, it suffices to establish the existence of $B^\nu$ that satisfies \eqref{finiteness} for $(x, y)\in\bR\times E$ with $L^\pi_t$ and $R^\pi_t$ replaced by $L^0_t$ and $R^0_t$, respectively. 
By the definition of $L^0$, \eqref{periodicMarkov} and \eqref{200_01}, we have, for $(x, y)\in\bR \times E$,
\begin{align}
&\bfE_{(x, y)}\sbra{\int_{[0,\infty)}e^{-\frq(t)} \diff L^0_t}
\\
&=x\lor0+ \sum_{k\in\bN}\bfE_{(x, y)}\sbra{e^{-\frq(T^\nu_k)} 
\cbra{\rbra{X_{T^\nu_k}-X_{T^\nu_{k-1}}}-\inf_{s \in [T^\nu_{k-1},T^\nu_{k}]}
\rbra{X_{s}-X_{T^\nu_{k-1}}}}
}
\\
&\leq x\lor0+ \sum_{k\in\bN}\bfE_{(x, y)}\sbra{e^{-\frq(T^\nu_{k-1})} 
\bfE_{(X_{T^\nu_{k-1}}, Y_{T^\nu_{k-1}})}\sbra{e^{-q T^\nu}\rbra{\rbra{X_{T^\nu}-X_0}-\inf_{s \in [0,T^\nu]}
\rbra{X_{s}-X_{0}}}}
}
\\
&=x\lor0+ \sum_{k\in\bN}\bfE_{(x, y)}\sbra{e^{-\frq(T^\nu_{k-1})} 
\bfE_{(0, Y_{T^\nu_{k-1}})}\sbra{e^{-q T^\nu}\rbra{X_{T^\nu}-
\un{X}_{T^\nu}
}}
}\\
&\leq x \lor0+\sum_{k\in\bN} \rbra{\int_0^\infty e^{-qt} \nu (\diff t)}^{k-1}
M_L, \label{bound_L}
\end{align}
where $\un{X}_t:= \inf_{s \in[0, t]} X_s$ for $t\geq 0$ and  
$M_L:= \sup_{y \in E} \bfE_{(0, y)}\sbra{e^{-q T^\nu}\rbra{X_{T^\nu}-
\un{X}_{T^\nu}
}}$. 
By a similar computation, we have, for $(x, y)\in\bR \times E$, 
\begin{align}
\bfE_{(x, y)}\sbra{\int_{[0,\infty)}e^{-\frq(t)} \diff R^0_t}
&\leq -(x\land 0)+\sum_{k\in\bN}\bfE_{(x, y)}\sbra{e^{-\frq(T^\nu_{k-1})} 
\bfE_{(0, Y_{T^\nu_{k-1}})}\sbra{\int_{(0, T^\nu]} 
e^{-qt} \diff \rbra{-\un{X}_t
}}
}
\\
&\leq -(x\land 0)+\sum_{k\in\bN} \rbra{\int_0^\infty e^{-qt} \nu (\diff t)}^{k-1}
M_R, \label{bound_R}
\end{align}
where $M_R:= \sup_{y \in E} \bfE_{(0, y)}\sbra{ \int_{(0, T^\nu]} 
e^{-qt} \diff \rbra{-\un{X}_{t}
}
}$. 
Based on the above discussion, it suffices to prove $M_L<\infty$ and $M_R<\infty$ and define $B^\nu:=(M_L\lor M_R)\sum_{k\in\bN} \rbra{\int_0^\infty e^{-qt} \nu (\diff t)}^{k-1}$. 
By the Markov property and \eqref{finite_ass}, we have, for $y \in E$, 
\begin{align}
\bfE_{(0, y)}\sbra{e^{-q T^\nu}\rbra{X_{T^\nu}-\un{X}_{T^\nu}
}} &\leq 2
\bfE_{(0, y)}\sbra{e^{-q T^\nu}
\sup_{s \in [0 , T^\nu]}|X_s|}
\\
&= 2\sum_{k\in\bN} 
\bfE_{(0, y)}\sbra{e^{-q T^\nu}
\sup_{s \in [0 , T^\nu]}|X_s|; T^\nu\in (k-1, k] } \\
&\leq 2\sum_{k\in\bN} e^{-q(k-1)} 
\bfE_{(0, y)}\sbra{
\sup_{s \in [0 , k]}|X_s|}
\nu((k-1, k])
\\
&\leq 
2\sum_{k\in\bN} e^{-q(k-1)} 
kM
\nu((k-1, k])
\leq2M\frac{e^{q-1}}{q} \nu ((0, \infty))
< \infty, \label{upper_sup}
\end{align}
which implies $M_L<\infty$. 
Similarly, by the Markov property and \eqref{finite_ass}, 
we have, for $y \in E$,
\begin{align}
\bfE_{(0, y)}\sbra{ \int_{(0, T^\nu]} e^{-qt} \diff \rbra{-\un{X}_{t}
}} 
&\leq \sum_{k\in\bN} 
\bfE_{(0, y)}\sbra{ \int_{(k-1, k]} e^{-qt} \diff \rbra{-\inf_{s \in [k-1, t]}(X_s-X_{k-1})}
} 
\\
&\leq\sum_{k\in\bN} e^{-q(k-1)} \sup_{y\in E}
\bfE_{(0, y)}\sbra{ \int_{(0, 1]} e^{-qt} \diff \rbra{-\un{X}_t
}} 
\\
&\leq M \sum_{k\in\bN} e^{-q(k-1)} 
=M\frac{1}{1-e^{-q}} 
< \infty, 
\end{align}
which implies $M_R<\infty$. 
The proof is complete. 
\end{proof} 
\begin{proof}[Proof of Lemma \ref{Lem308}]
For $n \in\bN$ and $(x, y)\in\bR\times E$, we have
\begin{align}
&\absol{v_\pi(x, y)-v_{\pi_n}(x, y)}\\
&=\absol{
\bfE_{(x, y)}\sbra{\int_{[T^\nu_{n+1}, \infty)}e^{-\frq(t)} \diff L^{\pi}_t 
  -\beta\int_{[T^\nu_{n+1}, \infty)}e^{-\frq(t)} \diff R^{\pi}_t+\beta\int_{[T^\nu_{n+1}, \infty)}e^{-\frq(t)} \diff R^{\pi_n}_t}}\\
& 
\begin{aligned}
\leq&{
\bfE_{(x, y)}\sbra{\int_{[T^\nu_{n+1}, \infty)}e^{-\frq(t)} \diff L^{\pi}_t}} 
+\beta{
\bfE_{(x, y)}\sbra{\int_{[T^\nu_{n+1}, \infty)}e^{-\frq(t)} \diff R^{\pi}_t 
}}+\beta{
\bfE_{(x, y)}\sbra{\int_{[T^\nu_{n+1}, \infty)}e^{-\frq(t)} \diff R^{\pi_n}_t}},
\end{aligned}
\label{Lem308_001}
\end{align}
where the equality follows from \eqref{def_pi_n}. 
By 
\eqref{upperbound_by_0strategy}, \eqref{bound_R} and the dominated convergence theorem, we have 
\begin{align}
\bfE_{(x, y)}\sbra{\int_{[T^\nu_{n+1}, \infty)}e^{-\frq(t)} \diff R^{\pi_n}_t}
&\leq\bfE_{(x, y)}\sbra{\int_{[T^\nu_{n+1}, \infty)}e^{-\frq(t)} \diff R^{0}_t}
\to 0 \qquad\text{ as }n\to\infty .
\label{Lem308_002}
\end{align}
By \eqref{finiteness} and the dominated convergence theorem, 
we have 
\begin{align}
\bfE_{(x, y)}\sbra{\int_{[T^\nu_{n+1}, \infty)}e^{-\frq(t)} \diff L^{\pi}_t}&
\to 0\quad\text{ as }n\to \infty, \label{Lem308_003}
\\
\bfE_{(x, y)}\sbra{\int_{[T^\nu_{n+1}, \infty)}e^{-\frq(t)} \diff R^{\pi}_t}&
\to 0\quad\text{ as }n\to \infty.  \label{Lem308_004}
\end{align}
By \eqref{Lem308_001}, \eqref{Lem308_002}, \eqref{Lem308_003} and \eqref{Lem308_004}, the proof is complete. 
\end{proof}
\begin{proof}[Proof of ($\Rightarrow$) in Lemma \ref{Lem311}]
We show that strategies in $\un{\Pi}^\nu$ not included in $\Pi^{\nu,\ast}_{V^\nu}$ are not optimal. 
Assume that a strategy $\pi\in \un{\Pi}^\nu$ does not belong to $\Pi^{\nu,\ast}_{V^\nu}$. 
We define 
\begin{align}
n_{\neq} := \min\cbra{k\in\bN: 1_{\{U^\pi_{T^\nu_k}> \bar{\bfb}_{V^\nu}(Y_{T^\nu_k})\}}+
1_{\{U^\pi_{T^\nu_k}< \un{\bfb}_{V^\nu}(Y_{T^\nu_k})\}}1_{\{\Delta L^\pi_{T^\nu_k}>0\}}>0
}. 
\end{align}
By the definition of $\pi$, $T^\nu_{n_{\neq}}<\infty$ with positive probability for some $\bfP_{(x, y)}$ with $(x, y)\in\bR\times E$. 
We fix such a $(x, y)\in\bR\times E$ and prove that 
\begin{align}
v_\pi(x, y)<V^\nu(x, y). \label{Sec00B_005}
\end{align}
We consider which continuation strategy is optimal after following $\pi$ up to time $T^\nu_{n_{\neq}}$. 
First, 
we assume that dividend payments according to $\pi$ are allowed only at the next $k\in\bN$ opportunities after $T^\nu_{n_{\neq}}$, that no dividends are paid thereafter, and that capital injections are made by reflection at $0$. 
Then, the expected NPV of dividends and capital injections is
\begin{align}
u_\pi(k, 0):=
\bfE_{(x, y)}&\Bigg{[}\int_{[0,T^\nu_{n_{\neq}+k-1}]}e^{-\frq(t)} \diff L^\pi_t-\beta\int_{[0,T^\nu_{n_{\neq}+k}]}e^{-\frq(t)} \diff R^\pi_t\\
&+e^{-\frq(T^\nu_{n_{\neq}+k})}
\rbra{ \Delta L^\pi_{T^\nu_{n_{\neq}+k}}+v_0(U^{\pi}_{T^\nu_{n_{\neq}+k}}, Y_{T^\nu_{n_{\neq}+k}})}\Bigg{]}. 
\end{align}
By \eqref{ineqg}, $u_\pi(k, 0)$ is not greater than
\begin{align}
u_\pi&(k, 1)\\
&:=
\bfE_{(x, y)}\sbra{\int_{[0,T^\nu_{n_{\neq}+k-1}]}e^{-\frq(t)} \diff L^\pi_t-\beta\int_{[0,T^\nu_{n_{\neq}+k}]}e^{-\frq(t)} \diff R^\pi_t
+e^{-\frq(T^\nu_{n_{\neq}+k})}
\tilde{v}_0(U^{\pi}_{({n_{\neq}+k}-)}, Y_{T^\nu_{n_{\neq}+k}})}\\
&=\bfE_{(x, y)}\Bigg{[}\int_{[0,T^\nu_{n_{\neq}+k-2}]}e^{-\frq(t)} \diff L^\pi_t-\beta\int_{[0,T^\nu_{n_{\neq}+k-1}]}e^{-\frq(t)} \diff R^\pi_t\\
&\qquad+e^{-\frq(T^\nu_{n_{\neq}+k-1})}
\rbra{\Delta L^{\pi}_{T^\nu_{n_{\neq}+k-1}}+\cV^{\nu,(1)}v_0(U^{\pi}_{T^\nu_{n_{\neq}+k-1}}, Y_{T^\nu_{n_{\neq}+k-1}})}\Bigg{]},
\end{align}
where the last equality follows from \eqref{periodicMarkov} and \eqref{optimal_b}. 
After repeating the same argument $k-1$ more times, we find that $u_\pi(k, 0)$ is not greater than
\begin{align}
u_\pi(k, k):=\bfE_{(x, y)}&\Bigg{[}\int_{[0,T^\nu_{n_{\neq}-1}]}e^{-\frq(t)} \diff L^\pi_t-\beta\int_{[0,T^\nu_{n_{\neq}}]}e^{-\frq(t)} \diff R^\pi_t\\
&+e^{-\frq(T^\nu_{n_{\neq}})}
\rbra{\Delta L^{\pi}_{T^\nu_{n_{\neq}}}+\cV^{\nu,(k)}v_0(U^{\pi}_{T^\nu_{n_{\neq}}}, Y_{T^\nu_{n_{\neq}}})}\Bigg{]}.
\end{align}
By taking the limit as $k\to\infty$, Lemmas \ref{Lem308}, 
Fatou's lemma and Lemma \ref{Lem309}, we have 
\begin{align}
v_\pi(x,y)= \lim_{k\to\infty}u_\pi(k, 0)\leq \limsup_{k\to\infty}u_\pi(k, k)
\leq  u_\pi(\infty),
\label{Sec00B_001} 
\end{align} 
where 
\begin{align}
u_\pi(\infty):=\bfE_{(x, y)}&\Bigg{[}\int_{[0,T^\nu_{n_{\neq}-1}]}e^{-\frq(t)} \diff L^\pi_t-\beta\int_{[0,T^\nu_{n_{\neq}}]}e^{-\frq(t)} \diff R^\pi_t\\
&+e^{-\frq(T^\nu_{n_{\neq}})}
\rbra{\Delta L^{\pi}_{T^\nu_{n_{\neq}}}+V^\nu(U^{\pi}_{T^\nu_{n_{\neq}}}, Y_{T^\nu_{n_{\neq}}})}\Bigg{]}. 
\end{align}
By the definition of $T^\nu_{n_{\neq}}$ and recalling \eqref{Lem305_001}, we have on $\{T^\nu_{n_{\neq}} <\infty\}$,  
 \begin{align}
\tilde{V}^\nu(U^{\pi}_{({n_{\neq}}-)}, Y_{T^\nu_{n_{\neq}}})& 
=
g^{U^{\pi}_{(n_{\neq}-)} }_{V^\nu}((U^{\pi}_{(n_{\neq}-)} -\un{\bfb}_{V^\nu} (Y_{T^\nu_{n_{\neq}}}))\lor 0, Y_{T^\nu_{n_{\neq}}}) \\
 &>
g^{U^{\pi}_{(n_{\neq}-)} }_{V^\nu}(\Delta L^{\pi}_{T^\nu_{n_{\neq}}}, Y_{T^\nu_{n_{\neq}}}) 
=\Delta L^{\pi}_{T^\nu_{n_{\neq}}}+V^\nu(U^{\pi}_{T^\nu_{n_{\neq}}}, Y_{T^\nu_{n_{\neq}}}). \label{Sec00B_002}
 \end{align}
Since $\bfP_{(x,y)}(T^\nu_{n_{\neq}}<\infty)>0$ and by \eqref{Sec00B_001} and \eqref{Sec00B_002}, we have 
\begin{align}
v_\pi(x,y)<\bfE_{(x, y)}&\Bigg{[}\int_{[0,T^\nu_{n_{\neq}-1}]}e^{-\frq(t)} \diff L^\pi_t-\beta\int_{[0,T^\nu_{n_{\neq}}]}e^{-\frq(t)} \diff R^\pi_t
+e^{-\frq(T^\nu_{n_{\neq}})}
\tilde{V}^\nu(U^{\pi}_{({n_{\neq}}-)}, Y_{T^\nu_{n_{\neq}}})\Bigg{]}. \ \ \ \
\label{Sec00B_003}
\end{align}
We define a strategy $\pi^\prime$ as the strategy that pays according to $\pi$ up to just before time $T^\nu_{n_{\neq}}$, and according to the MMPCB strategy with barrier $\bar{\bfb}_{V^\nu}$ from time $T^\nu_{n_{\neq}}$.
Then, using the definitions of $\bar{\bfb}_{V^\nu}$ and $\tilde{V}^\nu$, \eqref{periodicMarkov} and the fact that $\pi^\prime  $ belongs to $\Pi^{\nu,\ast}_{V^\nu}$, we have 
\begin{align}
\bfE_{(x, y)}&\sbra{\int_{[0,T^\nu_{n_{\neq}-1}]}e^{-\frq(t)} \diff L^\pi_t-\beta\int_{[0,T^\nu_{n_{\neq}}]}e^{-\frq(t)} \diff R^\pi_t
+e^{-\frq(T^\nu_{n_{\neq}})}
\tilde{V}^\nu(U^{\pi}_{(n_{\neq}-)}, Y_{T^\nu_{n_{\neq}}})}
\\
&=\bfE_{(x, y)}\sbra{\int_{[0,T^\nu_{n_{\neq}}]}e^{-\frq(t)} \diff L^{\pi^\prime}_t-\beta\int_{[0,T^\nu_{n_{\neq}}]}e^{-\frq(t)} \diff R^{\pi^\prime}_t
+e^{-\frq(T^\nu_{n_{\neq}})}
V^\nu(U^{\pi^\prime}_{T^\nu_{n_{\neq}}}, Y_{T^\nu_{n_{\neq}}})}\\
&=\sum_{k\in\bN}\bfE_{(x, y)}\sbra{\int_{[0,T^\nu_{k}]}e^{-\frq(t)} \diff L^{\pi^\prime}_t-\beta\int_{[0,T^\nu_{k}]}e^{-\frq(t)} \diff R^{\pi^\prime}_t+e^{-\frq(T^\nu_{k})}
V^\nu(U^{\pi^\prime}_{T^\nu_{k}}, Y_{T^\nu_{k}});n_{\neq}=k}\\
&\ \ +\bfE_{(x, y)}\sbra{\int_{[0,\infty)}e^{-\frq(t)} \diff L^{\pi^\prime}_t-\beta\int_{[0,\infty)}e^{-\frq(t)} \diff R^{\pi^\prime}_t;n_{\neq}=\infty}\\
&=\sum_{k\in\bN}\bfE_{(x, y)}\sbra{\int_{[0,\infty)}e^{-\frq(t)} \diff L^{\pi^\prime}_t-\beta\int_{[0,\infty)}e^{-\frq(t)} \diff R^{\pi^\prime}_t;n_{\neq}=k}\\
&\ \ +\bfE_{(x, y)}\sbra{\int_{[0,\infty)}e^{-\frq(t)} \diff L^{\pi^\prime}_t-\beta\int_{[0,\infty)}e^{-\frq(t)} \diff R^{\pi^\prime}_t;n_{\neq}=\infty}\\
&=v_{\pi^\prime}(x, y)= V^\nu (x, y).  \label{Sec00B_004}
\end{align}
By \eqref{Sec00B_003} and \eqref{Sec00B_004}, we obtain \eqref{Sec00B_005}. Hence 
$\pi$ is not optimal. The proof is complete. 
\end{proof}
\begin{proof}[Proof of Lemma \ref{right_derivative_lemma}]
Fix a non-negative measurable function $\bfb$ on $E$ and $(x, y)\in\bR\times E$. 
Using $(\sharp)$, we have 
\begin{align}
\varepsilon\bfE_{(x,y )}\sbra{ e^{-\frq({T}^{ (0), +}_{\bfb-})} ;{T}^{ (0), +}_{\bfb-} \leq {\tau}^{ (0), -}_0}
\leq&\bfE_{(x,y )}\sbra{\int_{[0,\infty)} e^{-\frq(t)} \diff (L^{{\pi}^{(\varepsilon)}_{\bfb}}_t-L^{{\pi}^{(0)}_{\bfb}}_t)}\\
\leq&
\varepsilon\bfE_{(x,y )}\sbra{ e^{-\frq({T}^{ (\varepsilon), +}_\bfb)} ;{T}^{ (\varepsilon), +}_\bfb < {\tau}^{ (\varepsilon), -}_{0+}},
\end{align}
and 
\begin{align}
-\varepsilon\bfE_{(x,y )}\sbra{ e^{-\frq({\tau}^{ (0), -}_0)} ;{\tau}^{ (0), -}_0<{T}^{ (0), +}_{\bfb-} }
&\leq
\bfE_{(x,y )}\sbra{\int_{(0,\infty)} e^{-\frq(t)} \diff ( R^{{\pi}^{(\varepsilon)}_{\bfb}}_t-R^{{\pi}^{(0)}_{\bfb}}_t)}\\
&\leq 
-\varepsilon\bfE_{(x,y )}\sbra{ e^{-\frq({\tau}^{ (\varepsilon), -}_{0+})} ;{\tau}^{ (\varepsilon), -}_{0+}< {T}^{ (\varepsilon), +}_\bfb }
.
\end{align}
Therefore, we have, for $ (x, y)  \in \bR \times E$,
\begin{align}
v^\prime_{{\pi}^{\nu}_{\bfb}+}(x, y)&=\lim_{\varepsilon\downarrow0}
\frac{v_{{\pi}^{\nu}_{\bfb}}(x+\varepsilon, y)-v_{{\pi}^{\nu}_{\bfb}}(x, y)}{\varepsilon}
=\lim_{\varepsilon\downarrow0}
\frac{v_{{\pi}^{(\varepsilon)}_{\bfb}}(x, y)-v_{{\pi}^{(0)}_{\bfb}}(x, y)}{\varepsilon}\\
&=\lim_{\varepsilon\downarrow0}
\frac{1}{\varepsilon}
\bfE_{(x,y )}\sbra{\int_{[0,\infty)} e^{-\frq(t)} \diff ( L^{{\pi}^{(\varepsilon)}_{\bfb}}_t-L^{{\pi}^{(0)}_{\bfb}}_t)
-\beta \int_{[0,\infty)} e^{-\frq(t)} \diff ( R^{{\pi}^{(\varepsilon)}_{\bfb}}_t- R^{{\pi}^{(0)}_{\bfb}}_t)}
\\
&=\bfE_{(x,y )}\sbra{ e^{-\frq({T}^{ \nu, +}_{\bfb-})} ;{T}^{ \nu, +}_{\bfb-} \leq {\tau}^{  -}_0}+
\beta\bfE_{(x,y )}\sbra{ e^{-\frq({\tau}^{  -}_0)} ;{\tau}^{  -}_0<{T}^{ \nu, +}_{\bfb-} }.
\end{align}
The proof is complete. 
\end{proof}

\begin{proof}[Proof of Lemma \ref{Lem312}]
By the version of Corollary \ref{Thm302} in which $\hat{\Xi}^\nu$ is replaced by $\hat{\Xi}^\nu_{V^\nu}$, 
the MMPCB strategy with barrier $\un{\bfb}_{V^\nu}$ is optimal. 
By this fact and \eqref{right_derivative}, we have, for $(x, y) \in [0, \infty) \times E$, 
\begin{align}
V^{\nu \prime}_+(x, y)
=&\lim_{\varepsilon\downarrow0}\frac{V^{\nu}(x+\varepsilon, y)-V^{\nu}(x, y)}{\varepsilon}
=\lim_{\varepsilon\downarrow0}\frac{v_{\pi^\nu_{\un{\bfb}_{V^\nu}} }(x+\varepsilon, y)-v_{\pi^\nu_{\un{\bfb}_{V^\nu}} }(x, y)}{\varepsilon}\\
=&\bfE_{(x, y)}\sbra{e^{-\frq(T^{\nu,+}_{\un{\bfb}_{V^\nu}-})};T^{\nu,+}_{\un{\bfb}_{V^\nu}-}\leq\tau^-_0}
+\beta \bfE_{(x, y)}\sbra{e^{-\frq(\tau^-_0)};\tau^-_0<T^{\nu,+}_{\un{\bfb}_{V^\nu}-}}\\
 \leq &\bfE\sbra{e^{-q T^\nu}}
+\beta \bfE_{(x, y)}\sbra{e^{-q\tau^-_0}}. \label{Lem312_001}
\end{align}
By \eqref{Lem312_001}, \eqref{200_01} and the dominated convergence theorem, we have, for $y \in E$, 
\begin{align}
\limsup_{x\to\infty}V^{\nu \prime}_+(x, y)\leq
\bfE\sbra{e^{-q T^\nu}}<1,
\end{align}
and thus $\bar{\bfb}_{V^\nu}(y)<\infty$. The proof is complete. 
\end{proof}
\begin{proof}[Proof of Lemma \ref{Lem317}]
We define the process $M:=\{M_t:t\geq 0\}$ as 
\begin{align}
M_t:= e^{-\frq(t)-rt }1_{\{t<\tau^-_{0+} \}},\qquad t\geq 0. 
\end{align}
Then, $M$ is a decreasing exact multiplicative functional (for the definition of a decreasing exact multiplicative functional, see, e.g., \cite[pp.259--260]{Sha1988}). 
Let us define, for any non-negative measurable function $f$, $t,\alpha\geq 0$, and $(x,y)\in(0, \infty)\times E$, 
\begin{align}
Q_t f(x, y):=\bfE_{(x, y)}\sbra{f(X_t, Y_t)M_t }, \qquad  
\end{align}
and 
\begin{align}
U^{(\alpha)}f(x, y):=\bfE_{(x, y)}\sbra{\int_0^\infty e^{-\alpha t}f(X_t, Y_t)M_t \diff t}. 
\end{align}
Then, $\{Q_t: t\geq 0\}$ is a sub-Markov semigroup with resolvent $\{U^{(\alpha)}: \alpha\geq 0\}$ by \cite[(56.6)]{Sha1988}. 
\par
The function $\rho^{\nu, \bfb}_{p}$ is excessive with respect to $\{Q_t: t\geq 0\}$ since for $(x,y)\in(0, \infty)\times E$ and $t\geq 0$, we have
\begin{align}
Q_t \rho^{\nu, \bfb}_{p}(x, y)= &
\bfE_{(x, y)}\sbra{e^{-\frq(t)-rt}1_{\{t< \tau^-_{0+}\}}\rho^{\nu, \bfb}_{p}(X_t, Y_t)}
=
\bfE_{(x, y)}\sbra{e^{-\frq(t)}1_{\{t< T^\nu\land\tau^-_{0+}\}}\rho^{\nu, \bfb}_{p}(X_t, Y_t)}
\\
=&\beta\bfE_{(x, y)}\sbra{e^{-\frq(\kappa^{\nu,\bfb}_{p,0} )}; t<T^\nu \land \tau^-_{0+}},
\end{align}
by the Markov property at $t$ and the memoryless property of $T^\nu$, and thus  
\begin{align}
Q_t \rho^{\nu, \bfb}_{p}(x, y)\leq \rho^{\nu, \bfb}_{p}(x, y),\qquad t\geq 0,
\end{align}
and 
\begin{align}
\lim_{t\downarrow0}Q_t \rho^{\nu, \bfb}_{p}(x, y)=\lim_{t\downarrow0} \beta\bfE_{(x, y)}\sbra{e^{-\frq(\kappa^{\nu,\bfb}_{p,0} )}; t<T^\nu \land \tau^-_{0+}}= \rho^{\nu, \bfb}_{p}(x, y). 
\end{align}
Therefore, the function $\rho^{\nu, \bfb}_{p}$ is excessive with respect to $\{U^{(\alpha)}: \alpha\geq 0\}$, so by (iv) of (56.13) in \cite{Sha1988}, the map $t\mapsto \rho^{\nu, \bfb}_{p} (X_t, Y_t)$ is right-continuous on $[0, \tau^-_{0+})$, $\bfP_{(x, y)}$-a.s. for any $(x, y)\in(0, \infty)\times E$. 
The proof is complete. 
\end{proof}

\begin{proof}[Proof of Lemma \ref{LemB01a}]
We fix $x_1, x_2 \in\bR$ with $x_1< x_2$. 
We denote by $\kappa^{(x),\bfb}_{p, 0}$ the counterpart of $\kappa^{\nu,\bfb}_{p, 0}$ for $X^{(x)}$ with $x \in \bR$. 
Here, we compare $\rho_{p}^{\nu,\bfb}( x_1, y)=\beta\bfE_{(0, y)}\sbra{e^{-\frq(\kappa^{(x_1),\bfb}_{p,0})}}$ and $\rho_{p}^{\nu,\bfb}( x_2, y)=\beta\bfE_{(0, y)}\sbra{e^{-\frq(\kappa^{(x_2),\bfb}_{p,0})}}$.  
By ($\sharp$), we have $U^{{\pi}^{(x_1)}_{\bfb}}_t\leq U^{{\pi}^{(x_2)}_{\bfb}}_t$
and $R^{{\pi}^{(x_1)}_{\bfb}}_t\geq R^{{\pi}^{(x_2)}_{\bfb}}_t$ for $t \geq 0$ and 
$U^{{\pi}^{(x_1)}_{\bfb}}_t- U^{{\pi}^{(x_2)}_{\bfb}}_t$
and $R^{{\pi}^{(x_1)}_{\bfb}}_t- R^{{\pi}^{(x_2)}_{\bfb}}_t$ are non-decreasing, 
$\bfP_{(0, y)}$-a.s. 
\par
On the event $\{\kappa^{(x_1),\bfb}_{p,0} \geq \kappa^{(x_2),\bfb}_{p,0}\}$, the following three cases exhaust all possibilities $\bfP_{(0,y)}$-a.s.:
\begin{enumerate}
\item
If $U^{{\pi}^{(x_1)}_{\bfb}}_{\kappa^{(x_2),\bfb}_{p,0}-} < U^{{\pi}^{(x_2)}_{\bfb}}_{\kappa^{(x_2),\bfb}_{p,0}-}$, then 
$R^{{\pi}^{(x_1)}_{\bfb}}$ increases on $[\kappa^{(x_2),\bfb}_{p,0},\kappa^{(x_2),\bfb}_{p,0}+\varepsilon)$ for any $\varepsilon>0$, and hence we obtain $\kappa^{(x_1),\bfb}_{p,0} \leq \kappa^{(x_2),\bfb}_{p,0}$. 
\item
If $U^{{\pi}^{(x_1)}_{\bfb}}_{\kappa^{(x_2),\bfb}_{p,0}-} = U^{{\pi}^{(x_2)}_{\bfb}}_{\kappa^{(x_2),\bfb}_{p,0}-}$ and $R^{{\pi}^{(x_2)}_{\bfb}}$ increases on $[\kappa^{(x_2),\bfb}_{p,0},\kappa^{(x_2),\bfb}_{p,0}+\varepsilon)$ for any $\varepsilon>0$, then $R^{{\pi}^{(x_1)}_{\bfb}}$ also increases on such intervals in the same manner as $R^{{\pi}^{(x_2)}_{\bfb}}$. 
Hence $\kappa^{(x_1),\bfb}_{p,0} \leq \kappa^{(x_2),\bfb}_{p,0}$. 
\item
If $U^{{\pi}^{(x_1)}_{\bfb}}_{\kappa^{(x_2),\bfb}_{p,0}-} = U^{{\pi}^{(x_2)}_{\bfb}}_{\kappa^{(x_2),\bfb}_{p,0}-}$ and $R^{{\pi}^{(x_2)}_{\bfb}}$ does not increase on $[\kappa^{(x_2),\bfb}_{p,0},\kappa^{(x_2),\bfb}_{p,0}+\varepsilon)$ for some $\varepsilon>0$, then  $L^{{\pi}^{(x_2)}_{\bfb}}$ must have increased at some time $\mu\in (0, \kappa^{(x_2),\bfb}_{p,0})$ for $U^{{\pi}^{(x_1)}_{\bfb}}$ and $U^{{\pi}^{(x_2)}_{\bfb}}$ to coincide. 
Thereafter, the same value of $\Lambda^p$ is used for both processes when deciding whether the hitting time of $0$ is included in the corresponding $\kappa^{(\cdot),\bfb}_{p,0}$. 
Hence $\kappa^{(x_1),\bfb}_{p,0} \leq \kappa^{(x_2),\bfb}_{p,0}$. 
\end{enumerate}
It follows from these cases that $\kappa^{(x_1),\bfb}_{p,0} \leq \kappa^{(x_2),\bfb}_{p,0}$ holds $\bfP_{(0,y)}$-a.s., and therefore $\rho_{p}^{\nu,\bfb}( x_1, y)\geq \rho_{p}^{\nu,\bfb}( x_2, y)$ holds for $y \in E$. 
The proof is complete.
\end{proof}
\begin{proof}[Proof of Lemma \ref{LemB01}]
We want to prove that there exists a measurable function $p_\bfb$ from $E$ to $[0, 1]$ such that
the identity in \eqref{LemB01_main} holds. 
If the identity in \eqref{LemB01_main} holds, then by using \eqref{periodicMarkov}, we have, for $(x, y) \in \bR \times E$,
\begin{align}
\rho^{\nu, \bfb}_{p_\bfb}&( x, y)\\
=&
\beta\bfE_{(x,y)}\sbra{e^{-\frq(\tau^-_0)};\tau^-_0<T^{\nu,+}_{\bfb-}}
+\sum_{k\in\bN}\beta\bfE_{(x,y)}\sbra{e^{-\frq(\kappa^{\nu,\bfb}_{p_\bfb,0})};T^{\nu,+}_{\bfb-}=T^\nu_k, T^\nu_k\leq\tau^-_0}
\\
=&
\beta\bfE_{(x,y)}\sbra{e^{-\frq(\tau^-_0)};\tau^-_0<T^{\nu,+}_{\bfb-}}
 +\sum_{k\in\bN}\bfE_{(x,y)}\sbra{e^{-\frq(T^\nu_k)}\varrho^{\nu, \bfb}_{p_\bfb}( \bfb(Y_{T^{\nu}_k}), Y_{T^{\nu}_k});T^{\nu,+}_{\bfb-}=T^\nu_k, T^\nu_k\leq\tau^-_0}
\\
=&
\beta\bfE_{(x,y)}\sbra{e^{-\frq(\tau^-_0)};\tau^-_0<T^{\nu,+}_{\bfb-}}
+\bfE_{(x,y)}\sbra{e^{-\frq(T^{\nu,+}_{\bfb-})}\varrho^{\nu, \bfb}_{p_\bfb}( \bfb(Y_{T^{\nu,+}_{\bfb-}}), Y_{T^{\nu,+}_{\bfb-}});T^{\nu,+}_{\bfb-}\leq\tau^-_0}\\
=&
\beta\bfE_{(x,y)}\sbra{e^{-\frq(\tau^-_0)};\tau^-_0<T^{\nu,+}_{\bfb-}}
+\bfE_{(x,y)}\sbra{e^{-\frq(T^{\nu,+}_{\bfb-})};T^{\nu,+}_{\bfb-}\leq\tau^-_0}. \label{LemB02a_001}
\end{align}
By Lemma \ref{LemB01a} and \eqref{LemB02a_001}, and since $\bfb$ belongs to $\Xi^\nu$, we obtain the inequality in \eqref{LemB01_main}. 
\par
(1) We now define candidates for $p_\bfb$ satisfying the required condition. 
For a measurable function $p$ and $y \in E$, we have
\begin{align}
\bfE_{(\bfb(y),y)} \sbra{ e^{-\frq(\cT^{1}_{p,\bfb})}; \cT^{1}_{p,\bfb}\leq\varkappa^{\nu,\bfb}_{p,0}}+\beta& \bfE_{(\bfb(y),y)} \sbra{ e^{-\frq(\varkappa^{\nu,\bfb}_{p,0})}; \varkappa^{\nu,\bfb}_{p,0}<\cT^{1}_{p,\bfb}}  \\
&=(1-p(y))\rho^{\nu,1}_{\bfb}(y) + p(y)\rho^{\nu,2}_{\bfb}(y). 
\label{LemB02_006}
\end{align}
Since $\rho^{\nu,1}_{\bfb}( y)\leq1$ and $\rho^{\nu,2}_{\bfb}( y)\geq 1$, the value
\begin{align}
p_\bfb (y) :=\inf\{a \in [0, 1]: (1-a)\rho^{\nu,1}_{\bfb}(y) + a\rho^{\nu,2}_{\bfb}(y)=1\}\in[0,1],
\end{align}
is well-defined for $y\in E$, and the function $p_\bfb$ is measurable and satisfies
\begin{align}
 (1-p_\bfb(y))\rho^{\nu,1}_{\bfb}(y) + p_\bfb(y)\rho^{\nu,2}_{\bfb}(y)=1,\qquad y\in E. 
\label{LemB02_007}
\end{align}
\par
(2) 
We define, for $n\in\bN$ and $y\in E$,
\begin{align}
\varrho^{\nu, \bfb}_{p_\bfb,(n)}( y):=&
\bfE_{(\bfb(y),y)} \sbra{ e^{-\frq(\cT^{1}_{p_\bfb,\bfb})}\varrho^{\nu, \bfb}_{p_\bfb,(n-1)}( Y_{\cT^{1}_{p_\bfb,\bfb}}); \cT^{1}_{p_\bfb,\bfb}\leq\varkappa^{\nu,\bfb}_{p_\bfb,0}}\\
&\qquad +\beta \bfE_{(\bfb(y),y)} \sbra{ e^{-\frq(\varkappa^{\nu,\bfb}_{p_\bfb,0})}; \varkappa^{\nu,\bfb}_{p_\bfb,0}<\cT^{1}_{p_\bfb,\bfb}},
\label{LemB02_001}
\end{align}
where $\varrho^{\nu, \bfb}_{p_\bfb,(0)}( y):=1$ for $y\in E$. 
We prove 
\begin{align}
\lim_{n\to\infty}\varrho^{\nu, \bfb}_{p_\bfb,(n)}( y)=\varrho^{\nu, \bfb}_{p_\bfb}( \bfb (y), y),\qquad y\in E. 
\label{LemB02_004}
\end{align}
We first prove that, for $n\in\bN$ and $y \in E$
\begin{align}
\varrho^{\nu, \bfb}_{p_\bfb,(n)}( y)-&\varrho^{\nu, \bfb}_{p_\bfb}( \bfb (y), y)\\
&=
\bfE_{(\bfb(y),y)} \sbra{ e^{-\frq(\cT^{n}_{p_\bfb,\bfb})}
\rbra{1-\varrho^{\nu, \bfb}_{p_\bfb}(\bfb( Y_{\cT^{n}_{p_\bfb,\bfb}}), Y_{\cT^{n}_{p_\bfb,\bfb}})}; \cT^{n}_{p_\bfb,\bfb}\leq\varkappa^{\nu,\bfb}_{p_\bfb,0}},
\label{LemB02_002}
\end{align}
by induction.
By using \eqref{periodicMarkov} at $\cT^{1}_{p_\bfb,\bfb}$ in the same way as \eqref{v_right_derivative}, we have, for $y\in E$, 
\begin{align}
\varrho^{\nu, \bfb}_{p_\bfb}( \bfb (y), y)=&
\bfE_{(\bfb(y),y)} \sbra{ e^{-\frq(\cT^{1}_{p_\bfb,\bfb})}\varrho^{\nu, \bfb}_{p_\bfb}(\bfb( Y_{\cT^{1}_{p_\bfb,\bfb}}), Y_{\cT^{1}_{p_\bfb,\bfb}}); \cT^{1}_{p_\bfb,\bfb}\leq\varkappa^{\nu,\bfb}_{p_\bfb,0}}\\
&\qquad\quad+\beta \bfE_{(\bfb(y),y)} \sbra{ e^{-\frq(\varkappa^{\nu,\bfb}_{p_\bfb,0})}; \varkappa^{\nu,\bfb}_{p_\bfb,0}<\cT^{1}_{p_\bfb,\bfb}}.
\label{LemB02_003}
\end{align}
By \eqref{LemB02_001} with $n=1$ and \eqref{LemB02_003}, we have \eqref{LemB02_002} with $n=1$. 
We assume that \eqref{LemB02_002} with $n=k\in\bN$ holds and prove \eqref{LemB02_002} with $n=k+1$. 
By \eqref{LemB02_001} with $n=k+1$ and \eqref{LemB02_003}, we have, for $y\in E$, 
\begin{align}
\varrho^{\nu, \bfb}_{p_\bfb,(k+1)}( y)-&\varrho^{\nu, \bfb}_{p_\bfb}( \bfb (y), y)\\
&=
\bfE_{(\bfb(y),y)} \sbra{ e^{-\frq(\cT^{1}_{p_\bfb,\bfb})}
({\varrho^{\nu, \bfb}_{p_\bfb,(k)}( Y_{\cT^{1}_{p_\bfb,\bfb}})-\varrho^{\nu, \bfb}_{p_\bfb}(\bfb( Y_{\cT^{1}_{p_\bfb,\bfb}}), Y_{\cT^{1}_{p_\bfb,\bfb}})}); \cT^{1}_{p_\bfb,\bfb}\leq\varkappa^{\nu,\bfb}_{p_\bfb,0}}. 
\end{align}
Applying \eqref{LemB02_002} with $n=k$ to the above equation and then using \eqref{periodicMarkov} at $\cT^{1}_{p_\bfb,\bfb}$ in the same way as \eqref{v_right_derivative}, we obtain \eqref{LemB02_002} with $n=k+1$. Therefore, \eqref{LemB02_002} holds for any $n\in\bN$ and $y\in E$. 
Since 
\begin{align}
\lim_{n\to\infty}\absol{\bfE_{(\bfb(y),y)} \sbra{ e^{-\frq(\cT^{n}_{p_\bfb,\bfb})}
\rbra{1-\varrho^{\nu, \bfb}_{p_\bfb}(\bfb( Y_{\cT^{n}_{p_\bfb,\bfb}}), Y_{\cT^{n}_{p_\bfb,\bfb}})}; \cT^{n}_{p_\bfb,\bfb}\leq\varkappa^{\nu,\bfb}_{p_\bfb,0}}}\qquad&\\
\leq
\lim_{n\to\infty}(1+\beta)\bfE_{(\bfb(y),y)} \sbra{ e^{-\frq(\cT^{n}_{p_\bfb,\bfb})}
; \cT^{n}_{p_\bfb,\bfb}\leq\varkappa^{\nu,\bfb}_{p_\bfb,0}}&=0, 
\end{align}
we obtain \eqref{LemB02_004}. 
\par
(3) We prove, for $n\in\bN$, 
\begin{align}
\varrho^{\nu, \bfb}_{p_\bfb,(n)}( y)=1, \qquad  y\in E, 
\label{LemB02_005}
\end{align}
by induction. 
By \eqref{LemB02_006}, \eqref{LemB02_007} and \eqref{LemB02_001} with $n=1$, 
we have \eqref{LemB02_005} with $n=1$. 
If \eqref{LemB02_005} with $n=k\in\bN$ holds, then $\varrho^{\nu, \bfb}_{p_\bfb,(k)}( Y_{\cT^{1}_{p_\bfb,\bfb}})$ in \eqref{LemB02_001} with $n=k+1$ can be replaced by $1$. 
Thus, by applying \eqref{LemB02_006} and \eqref{LemB02_007} to the resulting expression, we obtain \eqref{LemB02_005} with $n=k+1$. 
Therefore, we obtain \eqref{LemB02_005} for any $n\in\bN$. 
\par
By \eqref{LemB02_004} and \eqref{LemB02_005}, we obtain the identity in \eqref{LemB01_main}. 
The proof is complete. 
\end{proof}

\section{Proof of Proposition \ref{Prop303}}
\label{Prf_Prop303}
Let us define, for $(x, y)\in\bR\times E$,  
\begin{align}
\rho^{\nu,1}_{\bfb}(x, y)&:=\bfE_{(x,y)} \sbra{ e^{-\frq(T^{\nu,+}_{\bfb-})}; T^{\nu,+}_{\bfb-}\leq\tau^-_0}+\beta \bfE_{(x,y)} \sbra{ e^{-\frq(\tau^-_0)}; \tau^-_0<T^{\nu,+}_{\bfb-}}, 
\\
\rho^{\nu,2}_{\bfb}(x, y)&:=\bfE_{(x,y)} \sbra{ e^{-\frq(T^{\nu,+}_{\bfb})}; T^{\nu,+}_{\bfb}<\tau^-_{0+}}+\beta \bfE_{(x,y)} \sbra{ e^{-\frq(\tau^-_{0+})}; \tau^-_{0+}< T^{\nu,+}_{\bfb}}.  
\end{align}
Then, for $y\in E$, 
\begin{align}
x\mapsto \rho^{\nu,2}_{\bfb}(x, y)\text{ is left-continuous and }
\lim_{z\downarrow x}\rho^{\nu,2}_{\bfb}(z, y)=\rho^{\nu,1}_{\bfb}(x, y). 
\label{Prop303_003}
\end{align} 
\par
(1) We assume that $\bfb \in \hat{\Xi}^\nu$ and prove that $m_y (E\backslash \hat{E}^\bfb)=0$ for all $y \in E$.  
Let us define 
\begin{align}
\bfb_2 (y)=
\begin{cases}
\bfb (y) ,\qquad &y \in E^\bfb, \\
\un{\bfb}^\nu (y),\qquad &y \in E\backslash E^\bfb.
\end{cases}
\label{Prop303_008}
\end{align}
Then, $\bfb_2 \in \Xi^\nu$ and thus 
\begin{align}
\hat{E}^{\bfb_2}=E, \label{Prop303_001}
\end{align} 
by Theorem \ref{Thm302a}. 
Since $m_y (E\backslash E^\bfb)=0$ for all $y\in E$, we have 
\begin{align}
Y_{T^\nu_k} \in E^\bfb, \text{ and hence }\bfb(Y_{T^\nu_k})=\bfb_2(Y_{T^\nu_k}), \qquad k\in\bN,
\end{align}
$\bfP_{(x, y)}$-a.s. for all $(x, y)\in\bR\times E$. 
Thus, we have 
\begin{align}
\rho^{\nu,1}_{\bfb}(\bfb_2(y),  y)=
\rho^{\nu,1}_{\bfb_2}( y) ,\qquad 
\rho^{\nu,2}_{\bfb}(\bfb_2(y),  y)=
\rho^{\nu,2}_{\bfb_2}( y),\qquad y \in E. \label{Prop303_002}
\end{align}
By \eqref{Prop303_008}, \eqref{Prop303_001} and \eqref{Prop303_002}, we have 
$E^\bfb \subset\hat{E}^\bfb$ and hence $m_y (E\backslash \hat{E}^\bfb)=0$ for all $y \in E$. 

\par

(2) We assume $m_y (E\backslash \hat{E}^\bfb)=0$ for all $y \in E$ 
and prove that $\bfb \in \hat{\Xi}^\nu$. 
We define 
\begin{align}
\bfb_3 (y)=
\begin{cases}
\bfb (y) ,\qquad &y \in \hat{E}^\bfb, \\
\inf \{x \in [0, \infty) : \rho^{\nu,2}_{\bfb}(x, y) \leq 1\}   ,\qquad &y \in E\backslash \hat{E}^\bfb.
\end{cases}\label{Prop303_006}
\end{align}
Since $\lim_{x\to\infty}\rho^{\nu,2}_{\bfb}(x, y)=\bfE^Y_y\sbra{e^{-\frq (T^\nu)}}<1$, 
we have $\inf \{x \in [0, \infty) : \rho^{\nu,2}_{\bfb}(x, y) \leq 1\}<\infty$ for $y \in E$. 
Furthermore, by combining \eqref{Prop303_003} and \eqref{Prop303_006}, we obtain 
\begin{align}
\rho^{\nu,1}_{\bfb}(\bfb_3(y), y)\leq 1,\qquad
\rho^{\nu,2}_{\bfb}(\bfb_3(y ), y)\geq 1, \qquad y \in E. 
\label{Prop303_005}
\end{align}
Since $m_y (E\backslash \hat{E}^\bfb)=0$ for all $y\in E$, we have
\begin{align}
Y_{T^\nu_k} \in \hat{E}^\bfb, \text{ and hence }\bfb(Y_{T^\nu_k})=\bfb_3(Y_{T^\nu_k}), \qquad k\in\bN,
\end{align}
$\bfP_{(x, y)}$-a.s. for all $(x, y)\in\bR\times E$. Thus, we have 
\begin{align}
\rho^{\nu,1}_{\bfb}(\bfb_3(y), y)=\rho^{\nu,1}_{\bfb_3}( y) ,\qquad 
\rho^{\nu,2}_{\bfb}(\bfb_3(y), y)=\rho^{\nu,2}_{\bfb_3}( y),\qquad y \in E. \label{Prop303_004}
\end{align}
By \eqref{Prop303_005}, \eqref{Prop303_004} and Theorem \ref{Thm302a}, we have
\begin{align}
E^{\bfb_3}=E.\label{Prop303_007}
\end{align} 
By \eqref{Prop303_006} and \eqref{Prop303_007}, we have $\hat{E}^\bfb \subset E^\bfb$, and hence $m_y (E\backslash E^\bfb)=0$ for all $y \in E$. 
The proof is complete. 

\section{Proof of Theorem \ref{barrier_approximation_P}}
\label{Sec00C}
When $\nu$ is the exponential distribution with intensity $n$, 
we use notation obtained by replacing $\nu$ in Section \ref{Optimality_pcb} by $\rmP, n$. 
For $n\in\bN$ and $f\in\Gamma^{\rmP, n+1}$, we define $\cU^{\rmP, n}$ for $(x, y)\in \bR\times E$ by  
\begin{align}
\cU^{\rmP, n} f(x, y)=&\bfE_{(x,y)}\bigg{[}-\beta\int_{[0, T^{\rmP,n+1}]}e^{-\frq(t)}\diff ((-\un{X}_t\lor0))
\\
&+e^{-\frq(T^{\rmP,n+1})}\bigg{\{}1_{\{T^{\rmP,n+1}<T^{\rmP,n}\}}f(X^0_{T^{\rmP,n+1}}, Y_{T^{\rmP,n+1}})\\
&+
1_{\{T^{\rmP,n+1}=T^{\rmP,n}\}}
\rbra{\tilde{f}(X^0_{T^{\rmP,n+1}}, Y_{T^{\rmP,n+1}})}
\bigg{\}}
\bigg{]}. 
\end{align}
Focusing on condition (iii) in the definition of the class $\Gamma^\nu$, a slight extension of \eqref{upperbound_by_0strategy} shows that $\Gamma^{{\rmP, n}}\subset \Gamma^{{\rmP, n+1}}$ holds, and that $\cU^{\rmP, n}$ can also be applied to the functions in $\Gamma^{{\rmP, n}}$. 
By applying the same method as in the proof of Lemma \ref{Lem306}, we can verify that $\cU^{\rmP, n} f\in \Gamma^{\rmP, n+1}$ for $f \in \Gamma^{\rmP, n+1}$. 
We can define, for $k\in\bN$,
\begin{align}
U^{\rmP, n}_{(k)}(x, y):=\cU^{\rmP, n}(U^{\rmP, n}_{(k-1)}) (x, y),\qquad (x, y)\in\bR\times E,
\end{align}
inductively, where $U^{\rmP, n}_{(0)}(x, y):=v_0(x, y)$ for $(x, y)\in \bR\times E$.
For simplicity, $\cV^{\rmP, n, (k)} {v_0}$ will be denoted by $V^{\rmP, n}_{(k)} $ for $n\in\bN$ and $k\in\bN\cup\{0\}$. 
\par
By an argument similar to the proof of Lemma \ref{Lem309}, we obtain the following lemma. 
\begin{Lem}\label{LemC01}
For each $n\in\bN$, we have 
\begin{align}
\lim_{k\to\infty}U^{\rmP, n}_{(k)} (x, y)= V^{\rmP, n}(x, y),\qquad (x, y)\in\bR\times E. \label{LemB01_statement}
\end{align}
\end{Lem}
\begin{proof}
We fix $n\in\bN$. 
We prove, for $k\in\bN\cup\{0\}$, 
\begin{align}
U^{\rmP, n}_{ (k)} (x, y)\leq  V^{\rmP, n}(x, y),\qquad (x, y)\in\bR\times E. \label{LemC1_001}
\end{align}
When $k=0$, \eqref{LemC1_001} holds since $U^{\rmP, n}_{ (0)} (x, y)= {v_0}(x, y)$ for $(x, y)\in\bR\times E$. 
We assume that \eqref{LemC1_001} holds when $k=l\in\bN\cup\{0\}$ and prove that \eqref{LemC1_001} holds when $k=l+1$. 
By \eqref{optimal_b}, \eqref{periodicMarkov} at $T^{\rmP, n+1}$ and the memoryless property of the exponential distribution, we have
\begin{align}
\cV^{\rmP, n}& V^{\rmP,n}(x, y)\\
&=
\bfE_{(x,y)}\bigg{[}-\beta\int_{[0, T^{\rmP,n+1}]}e^{-\frq(t)}\diff ((-\un{X}_t\lor0))
\\
&+e^{-\frq(T^{\rmP,n+1})}\bigg{\{}1_{\{T^{\rmP,n+1}<T^{\rmP,n}\}}
\bigg{(}-\beta\int_{( T^{\rmP,n+1},T^{\rmP,n}]}e^{-\frq(t)}\diff ((-\un{X}_t\lor0))\\
&+ 
e^{-(\frq(T^{\rmP,n})-\frq(T^{\rmP,n+1}))}\tilde{V}^{\rmP,n}(X^0_{T^{\rmP,n}}, Y_{T^{\rmP,n}})\bigg{)}
+
1_{\{T^{\rmP,n+1}=T^{\rmP,n}\}}
\tilde{V}^{\rmP,n}(X^0_{T^{\rmP,n+1}}, Y_{T^{\rmP,n+1}})
\bigg{\}}
\bigg{]}
\\
=&
\bfE_{(x,y)}\bigg{[}-\beta\int_{[0, T^{\rmP,n+1}]}e^{-\frq(t)}\diff ((-\un{X}_t\lor0))
+e^{-\frq(T^{\rmP,n+1})}\bigg{\{}1_{\{T^{\rmP,n+1}<T^{\rmP,n}\}}
V^{\rmP,n}(X^0_{T^{\rmP,n+1}}, Y_{T^{\rmP,n+1}})\\
&+
1_{\{T^{\rmP,n+1}=T^{\rmP,n}\}}
\tilde{V}^{\rmP,n}(X^0_{T^{\rmP,n+1}}, Y_{T^{\rmP,n+1}})
\bigg{\}}
\bigg{]}
\\
=&\cU^{\rmP, n} V^{\rmP,n}(x, y).
\label{LemC01_002}
\end{align}
Applying the operator $\cU^{\rmP, n}$ to both sides of \eqref{LemC1_001} with $k=l$, and using 
 \eqref{Lem310_003} with $\nu$ replaced by $\rmP,n$ and  \eqref{LemC01_002}, we have \eqref{LemC1_001} with $k=l+1$. 
In conclusion, we have 
\begin{align}
\limsup_{k\to\infty}U^{\rmP, n}_{ (k)} (x, y)\leq  V^{\rmP, n}(x, y),\qquad (x, y)\in\bR\times E. \label{LemB01_main_001}
\end{align}
\par
Here, taking \eqref{optimal_b} into account and repeating the argument in the same manner as in \eqref{LemC01_002}, with \eqref{periodicMarkov} applied at $T^{\rmP, n+1}_1,\dots, T^{\rmP, n+1}_k$, we obtain the following interpretation of $U^{\rmP, n}_{ (k)}$. 
It is the value function for dividends and capital injections over the class of strategies in $\un{\Pi}^{\rmP, n}$ under the additional constraint that dividends may be paid only at the jump times of $\hat{N}^{(n)}$ that coincide with one of the first $k$ jump times of $\hat{N}^{(n+1)}$. 
We define $\pi^{{[k]}_{(n)}}_{\un{\bfb}^{\rmP, n}}\in\un{\Pi}^{\rmP, n}$ as a strategy that uses the same control as $\pi^{\rmP, n}_{\un{\bfb}^{\rmP, n}}$ until time $T^{\rmP, n+1}_k$, and does not pay dividends thereafter. 
Using the dominated convergence theorem as in the proof of Lemma \ref{Lem308}, we have 
\begin{align}
\lim_{k\to\infty}v_{\pi^{{[k]}_{(n)}}_{\un{\bfb}^{\rmP, n}}}(x, y)
=v_{\pi^{\rmP, n}_{\un{\bfb}^{\rmP, n}}}(x, y), \qquad (x, y) \in \bR \times E. 
\label{SecC002}
\end{align}
Also, in the proof of \eqref{induction_inequality} 
with $\nu$ replaced by $\rmP, n+1$, by replacing 
$\cV^{\rmP, n+1,  (k)}v_0$ with $U^{\rmP, n}_{(k)}$ and using the same indicator-based decomposition of expectations as in \eqref{LemC01_002}, we obtain 
\begin{align}
v_{\pi^{{[k]}_{(n)}}_{\un{\bfb}^{\rmP, n}}}(x, y) \leq U^{\rmP, n}_{(k)}(x,y),\qquad (x, y)\in\bR
\times E. \label{SecC001}
\end{align}
Thus, for $k\in\bN$ and $(x, y)\in\bR\times E$, we have 
\begin{align}
&V^{\rmP, n}(x, y)-U^{\rmP, n}_{ (k)} (x, y)\leq V^{\rmP, n}(x, y)-v_{\pi^{{[k]}_{(n)}}_{\un{\bfb}^{\rmP, n}}}(x, y)\\
&\leq V^{\rmP, n}(x, y)-v_{\pi^{\rmP, n}_{\un{\bfb}^{\rmP, n}}}(x, y)
+\absol{v_{\pi^{\rmP, n}_{\un{\bfb}^{\rmP, n}}}(x, y)-v_{\pi^{{[k]}_{(n)}}_{\un{\bfb}^{\rmP, n}}}(x, y)}
=\absol{v_{\pi^{\rmP, n}_{\un{\bfb}^{\rmP, n}}}(x, y)-v_{\pi^{{[k]}_{(n)}}_{\un{\bfb}^{\rmP, n}}}(x, y)}, 
\label{Lem309_002C}
\end{align}
where the first inequality follows from \eqref{SecC001}. 
By taking the limit of \eqref{Lem309_002C} as $k\to \infty$ and by \eqref{SecC002}, we have
\begin{align}
V^{\rmP, n}(x, y)-\liminf_{k\to\infty}U^{\rmP, n}_{ (k)} (x, y)
\leq0, \qquad (x, y)\in\bR\times E. \label{Lem309_005C}
\end{align}
By \eqref{LemB01_main_001} and \eqref{Lem309_005C}, we obtain \eqref{LemB01_statement}. 
The proof is complete. 
\end{proof}
\begin{proof}[Proof of Theorem \ref{barrier_approximation_P}]
As in the proof of Theorem \ref{barrier_approximation_1}, it suffices to prove that, for $n\in\bN$,  
\begin{align}
V^{\rmP, n\prime}_{+}(x, y)
\leq V^{\rmP, {n+1}\prime}_{+}(x, y),\qquad (x,y)\in \bR\times E, \label{barrier_result_P}
\end{align}
where $V^{\rmP, n\prime}_+$ is the right derivative of $V^{\rmP, n}$. 
\par
The functions $x\mapsto V^{\rmP, n}(x, y)$, $x\mapsto V^{\rmP, n+1}(x, y)$, $x\mapsto V^{\rmP, n}_{(k)}(x, y)$ and $x\mapsto U^{\rmP, n}_{(k)}(x, y)$ are concave for $n, k\in\bN$ and $y \in E$. 
Hence, by Lemma \ref{Lem309}, 
Lemma \ref{LemC01}, \cite[Theorem B.4.2.3]{HirLem2001} and \cite[Theorem 1.1]{Lac1982},  
we have 
$
\lim_{k\to\infty}{V^{\rmP, n+1\prime}_{(k)+}}(x, y)={V^{\rmP, n+1\prime}_+}(x, y)$ and $\lim_{k\to\infty}{U^{\rmP, n\prime}_{(k)+}}(x,  y)={V^{\rmP, n\prime}_+}(x, y)$ for Lebesgue-a.e. $x$, for any $n\in\bN$ and $y \in E$. 
Therefore, if we show that for $n \in \bN$, 
\begin{align}
U^{\rmP, n\prime}_{(k)+}(x, y)
\leq V^{\rmP, n+1\prime}_{(k)+}(x, y),\qquad k\in\bN\cup\{0 \},~(x,y)\in \bR\times E, \label{barrier_convergence_P}
\end{align}
then \eqref{barrier_result_P} follows.
In the following, we prove \eqref{barrier_convergence_P} by induction on $k$. 
\par 
For $k=0$, since both ${U^{\rmP, n}_{(0)}}$ and ${V^{\rmP,n+1}_{(0)}}$ equal $v_0$, \eqref{barrier_convergence_P} holds. 
Assuming that \eqref{barrier_convergence_P} holds for $k=l\in\bN\cup\{0\}$, we show that \eqref{barrier_convergence_P} also holds for $k=l+1$. 
By applying almost the same method as in the proof of \eqref{Lem306_004}, we have, 
for $f\in\Gamma^{\rmP, n+1}$ and $(x,y)\in\bR\times E$, 
\begin{align}
{(\cU^{\rmP, n}f)}^\prime_+ &(x, y)
= 
\bfE_{(x,y)}\bigg{[}e^{-\frq(T^{\rmP,n+1}\land\tau^-_0)}\bigg{\{}1_{\{T^{\rmP,n+1}<T^{\rmP,n}, \un{X}_{T^{\rmP,n+1} }\geq0\}}
f^\prime_{+}(X_{T^{\rmP,n+1}}, Y_{T^{\rmP,n+1}})\\
&+
1_{\{T^{\rmP,n+1}=T^{\rmP,n},\un{X}_{T^{\rmP,n+1} }\geq0\}}
\rbra{\tilde{f}^{\prime}_{+}(X_{T^{\rmP,n+1}}, Y_{T^{\rmP,n+1}})}
+1_{\{\un{X}_{T^{\rmP,n+1} }<0\}}\beta
\bigg{\}}
\bigg{]}. \label{strong_Markov_P}
\end{align}
By \eqref{Lem306_004} and \eqref{strong_Markov_P}, 
we have, for $f, g\in \Gamma^{\rmP, n+1}$ with $f^\prime_+(x, y)\leq g^\prime_+(x, y)$ for $(x,y)\in\bR\times E$,
\begin{align}
{(\cU^{\rmP, n}f)}^\prime_+(x ,y )\leq {(\cV^{\rmP, n+1}g)}^\prime_+(x ,y ),
\qquad (x,y)\in\bR\times E. 
\label{compare_periodic}
\end{align} 
By the definitions of $U^{\rmP, n}_{(l+1)}$ and $V^{\rmP, n+1}_{(l+1)}$, \eqref{barrier_convergence_P} with $k=l$ and \eqref{compare_periodic}, we have \eqref{barrier_convergence_P} with $k=l+1$.  
The proof is complete. 
\end{proof}


\bibliographystyle{jplain}
\bibliography{NOBA_references_06}

\end{document}